\documentclass[11pt]{amsart}

\usepackage{amsthm,amsmath,amssymb}

\usepackage{tikz}
\usepackage{graphicx}
\usepackage{diagbox}

\usepackage{color}
\usepackage{colortbl}
\newcommand\tX{\widetilde X}
\newcommand\X[1]{X^{(#1)}}
\newcommand\ra[1]{r_a^{(#1)}}
\newcommand\rb[1]{r_b^{(#1)}}
\newcommand\XOOX{{\left(\!\begin{array}{c}\begin{array}{c|c}X\ \ & 0 \end{array} \\ \hline \begin{array}{c|c}0 &\ \ X\end{array}\end{array}\!\right)}}

\definecolor{thingray}{rgb}{0.925,0.925,0.925}
\definecolor{gray}{rgb}{0.825,0.825,0.825}
\newcommand\m[1]{\cellcolor{gray}{#1}}
\newcommand\z[1]{\cellcolor{thingray}{#1}}
\newcommand\p[1]{{#1}}

\usepackage[margin=3cm]{geometry}

\RequirePackage{hyperref}
\RequirePackage{hypernat}

\newtheorem{theorem}{Theorem}[section]
\newtheorem{lemma}[theorem]{Lemma}
\newtheorem{proposition}[theorem]{Proposition}
\newtheorem{corollary}[theorem]{Corollary}

\theoremstyle{remark}
\newtheorem{example}{Example}[section]
\newtheorem{remark}{Remark}[section]
\theoremstyle{definition}
\newtheorem{definition}[theorem]{Definition}

\DeclareMathOperator{\tr}{tr}
\DeclareMathOperator{\diag}{diag}
\DeclareMathOperator{\rk}{rank}

\numberwithin{equation}{section}

\title[Existence of Kronecker Covariance MLE]{Existence and Uniqueness of the Kronecker\\ Covariance MLE}
\author{Mathias Drton}
\address{Department of Mathematics, Technical University of Munich, Boltzmannstra{\ss}e 3,
85748 Garching b. M\"unchen, Germany}
\author{Satoshi Kuriki}
\address{The Institute of Statistical Mathematics, 10-3 Midoricho, Tachikawa, Tokyo 190-8562, Japan}
\author{Peter Hoff} 
\address{Department of Statistical Science, Duke University, Durham, NC 27708-0251, United States}

\begin{document}

\begin{abstract}
In matrix-valued datasets the sampled matrices often exhibit correlations among both their rows and their columns. A useful and parsimonious model of such dependence is the matrix normal model, in which the covariances among the elements of a random matrix are parameterized in terms of the Kronecker product of two covariance matrices, one representing row covariances and one representing column covariance.  An appealing feature of such a matrix normal model is that the Kronecker covariance structure allows for standard likelihood inference even when only a very small number of data matrices is available. For instance, in some cases a likelihood ratio test of dependence may be performed with a sample size of one.  However, more generally the sample size required to ensure boundedness of the matrix normal likelihood or the existence of a unique maximizer depends in a complicated way on the matrix dimensions.  This motivates the study of how large a sample size is needed to ensure that maximum likelihood estimators exist, and exist uniquely with probability one.  Our main result gives precise sample size thresholds in the paradigm where the number of rows and the number of columns of the data matrices differ by at most a factor of two.  Our proof uses invariance properties that allow us to consider data matrices in canonical form, as obtained from the Kronecker canonical form for matrix pencils.
\end{abstract}






\maketitle



\section{Introduction}
\label{sec:introduction}

\subsection{Kronecker covariances and matrix normal models}
\label{subsec:intro:kronecker}

A matrix-valued dataset consists of a sample of matrices 
$Y_1,\ldots, Y_n$, each taking values in $\mathbb R^{m_1\times m_2}$
for integers $m_1, m_2\ge 2$. 
Such data arise in spatial statistics \cite{koch:2020} as well as in a variety of experimental settings where 
outcomes are obtained under combinations of two conditions, such as 
inter\-national trade between pairs of countries \cite{volfovsky:2015},
studies involving 
multivariate time-series of EEG measurements on multiple individuals 
\cite{makeig_2012}, 
age by period human mortality data \cite{fosdick_hoff_2014}, and 
factorial experiments arising in genomics \cite{allen_tibshirani_2012}, to name a few.  
In these applications the matrices $Y_1,\dots,Y_n$ often exhibit substantial covariance among their rows and covariance among their columns, 
in the sense that (after appropriately subtracting out any mean or regression effects) 
the two empirical covariance matrices 
$\tfrac{1}{n}\sum_{i=1}^n Y_i Y_i^T$ and $\tfrac{1}{n}\sum_{i=1}^n Y_i^T Y_i$ are substantially non-diagonal. 
This motivates the use of a statistical model that can represent 
such data features. 

One suitable and widely-used model is based on  
matrix normal distributions; see \cite{dawid_1981,ohlson:2013,greenewald,glanz:2018} or also the work on graphical modeling in \cite{yin:2012,zhu:2018,MR3889377,MR3210978,MR2758420}.  Let $Z$ be an $m_1\times m_2$ random matrix with i.i.d.\ standard normal entries.  A random matrix $Y$ taking values in $\mathbb R^{m_1\times m_2}$ is said to have 
a matrix normal distribution if there exist (deterministic) matrices $M\in \mathbb R^{m_1\times m_2}$, $A\in \mathbb R^{m_1\times m_1}$, and $B\in \mathbb R^{m_2\times m_2}$ such that $Y\stackrel{d}{=} M + A Z B^T$.  
In this case, we write $Y \sim \mathcal N(M, \Sigma_2\otimes \Sigma_1)$, 
where $\Sigma_1 = A A^T$ and $\Sigma_2 =B B^T$ and ``$\otimes$'' denotes the Kronecker product.  
This notation reflects that $\Sigma_2\otimes \Sigma_1$ is the covariance matrix of the vectorization of $Y$.  If $A$ and $B$, and thus also $\Sigma_1$ and $\Sigma_2$, are invertible then the matrix normal distribution $\mathcal N( M , \Sigma_2\otimes \Sigma_1)$ is regular and 
has a Lebesgue density.

For fixed $m_1$ and $m_2$, the matrix normal model is the set of all regular matrix normal distributions. 
The number of its covariance parameters is on the order of $m_1^2 +m_2^2$, 
which is a substantial reduction as compared to the normal model with unrestricted covariance and on the order of $m_1^2 m_2^2$ parameters. 
As a result, matrix normal models can be used to make 
likelihood-based inference with sample sizes that would preclude use of an unrestricted covariance model. 
For example, in \cite{volfovsky:2015}, matrix normal
models were used to construct non-trivial tests of dependence for square data matrices even when 
the sample size was one.

\subsection{Practical relevance of sample size conditions}
Numerical experiments indicate that the sample size required for a bounded likelihood 
or unique maximum likelihood estimator (MLE) depends in subtle ways on the dimensions $m_1$ and $m_2$ of the data matrices; see Example~\ref{ex:flipflop} below. 
Despite recent progress on 
sufficient sample size conditions, 
the
precise behavior of the matrix normal likelihood function in settings with small
sample size $n$ is not fully understood \cite{soloveychik:2016}.
Specifically, from prior literature it is not known when precisely the \emph{Kronecker MLE}, i.e.,
MLE of the covariance matrix $\Sigma_2\otimes\Sigma_1$,
exists or exists uniquely.  
This is the problem we consider in this article, which gives new sample size conditions that provide a full solution to the problem for matrices whose dimensions $m_1$ and $m_2$ differ by a factor of at most 2; this is the regime where prior results leave the largest gaps.
Precise sample size conditions are of great practical relevance 
as they are useful for the development of
numerical methods to obtain MLEs, the design of experiments, and as a guide to alternative data analysis strategies
when conditions for existence or uniqueness are not met.

Regarding numerical methods, matrix normal models generally do not admit a closed form MLE.
Hence, a data analyst relies on iterative methods of computation, for which suitable convergence criteria need to be set \cite{dutilleul:1999}.
However, as we will describe below, iteration of such an algorithm will either
 converge to a unique maximizer, converge to a non-unique maximizer, or  not converge. Absent knowledge of the sample size conditions for existence of the Kronecker MLE, at each iteration
the algorithm would have
to distinguish in an ad-hoc manner between a situation in which the procedure has not yet sufficiently converged but will converge eventually,
and one in which it will never converge. Additionally, in cases where the MLE exists, the analyst would certainly want to know if it exists uniquely, as non-uniqueness implies multiple explanations of the data generating mechanism that are equally valid (based on likelihood). Absent knowledge of the sample size conditions for uniqueness, the data analyst might not be aware that an MLE obtained via the iterative procedure was not unique. Even if the procedure was run from several different starting values, an assessment of non-uniqueness could be imprecise, as differences from different runs could be due to non-uniqueness, or just an indication that the algorithm
has not been sufficiently iterated. In contrast, precise
sample size conditions allow one to side-step these numerical issues.

Knowledge of sample size conditions can also assist with study design. For example, consider an experiment in which each replication produces a matrix of gene expression levels for a set of tissue types.
If interest is inferring dependencies among genes and among types
using the matrix normal model
(see, e.g., \cite{efron_2009}),
then knowledge of the sample size requirement for existence
and uniqueness of the Kronecker MLE would certainly be useful. Additionally, in cases where
the study has been completed and the sample size conditions have not been met, our results still provide a guide to alternative methods for making inference. For example, suppose we wish to test
 a null hypothesis $H: \Sigma_2= I$ of exchangeability within each row. If the sample size is insufficient for existence of the MLE, then the likelihood ratio statistic is infinity with probability one and a non-trivial
level-$\alpha$ likelihood ratio test is unavailable. However, suppose
 the  maximum number $m_2'$ of columns
for which the maximized
likelihood is bounded, and thus for which a
likelihood ratio test with non-trivial power may be obtained, were known.
In this case, the null hypothesis $H$  could be evaluated by properly combining the results of multiple tests of exchangeability on subsets of $m_2'$ columns, even at sample sizes as small as $n=1,2$.

\subsection{Maximum likelihood thresholds and known results}
\label{subsec:intro:thresholds}

For notational simplicity, we obtain results for samples of size $n$ from a 
 mean-zero matrix normal model, that is, 
\begin{equation}
\label{eq:kronecker-structure}
Y_1,\ldots, Y_n \sim \mbox{i.i.d.} \ \mathcal N(0,\Sigma_2\otimes \Sigma_1). 
\end{equation}
As noted in Remark~\ref{rem:mean-unknown}, 
sample-size thresholds for the case of an unknown mean will be equal to those in this 
mean-zero case, plus one additional data matrix. 
Let $\mathit{PD}(m)$ be the cone of positive definite $m\times m$
matrices.  Let $\Psi_1=\Sigma_1^{-1}$ and $\Psi_2=\Sigma_2^{-1}$ be
the precision matrices.  Ignoring an additive constant, two times the
log-likelihood function for the matrix normal model can be written as
\begin{equation}
\ell(\Psi_1,\Psi_2) = n m_2\log\det(\Psi_1) + n m_1\log\det(\Psi_2) -
\tr\bigg(\Psi_1 \sum_{i=1}^n
Y_i\Psi_2Y_i^T\bigg).
  \label{eq:log-likelihood}
\end{equation}
The log-likelihood function only depends on the Kronecker product
$\Psi_2\otimes \Psi_1$ and 
\begin{equation}
  \label{eq:scaling}
  \ell\left(c\Psi_1,c^{-1}\Psi_2\right) \;=\;\ell(\Psi_1,\Psi_2)
\end{equation}
for all scalars $c>0$.

A standard method to compute the MLE of $\Psi_2\otimes \Psi_1$ is
block-coordinate descent, also referred to as the ``flip-flop
algorithm'' \cite{dutilleul:1999}.  
If $\Psi_2$ is fixed to be a value $\tilde \Psi_2$, the likelihood function is maximized by $\tilde \Psi_1 = \tilde \Sigma_1^{-1}$ with $\tilde\Sigma_1 = \tfrac{1}{nm_2}\sum_i Y_i 
\tilde \Psi_2 Y_i^T$.  Similarly, since the trace term in \eqref{eq:log-likelihood} can be alternatively written 
as $\tr\bigl( \Psi_2 \sum_i Y_i^T \Psi_1 Y_i\bigr)$, the maximizer when fixing $\Psi_1=\tilde \Psi_1$ is $\tilde \Psi_2=\tilde\Sigma_2^{-1}$ with $\tilde \Sigma_2 = \tfrac{1}{nm_1}\sum_i Y_i^T \tilde \Psi_1 Y_i$. 
The flip-flop algorithm proceeds by iteratively updating $\tilde \Psi_1$ and 
$\tilde \Psi_2$ (or equivalently $\tilde \Sigma_1$ and $\tilde \Sigma_2$)
using these formulas. Evidently, the updates are well-defined 
as long as both 
$\sum_i Y_i 
\tilde \Psi_2 Y_i^T$  and 
$\sum_i Y_i^T \tilde \Psi_1 Y_i$ are invertible. As will be made precise 
later in Lemma~\ref{lem:profile}, this condition will be 
met a.s.~as long as the ``row sample size'' $n m_2$ is as big as the number of 
rows $m_1$, and similarly $n m_1$ is as big as $m_2$.

\begin{example}
  \label{ex:flipflop}
  Consider the case of $n=2$ data matrices of size $(m_1,m_2)$, where we fix $m_2=4$ and consider $m_1=5,6,7,8$.  In each case we take the data to be in a canonical form (as specified in Theorem~\ref{thm:tenberge}) and run the flip-flop algorithm starting from a random choice for $\Psi_1$.  Figure~\ref{fig:flipflop} depicts the
  behavior of the algorithm in terms of the function $g$ defined as
  $-2/n$ times the log-likelihood (solid line), where we also omitted additive constants from the log-likelihood; 
  see~\eqref{eq:profile-neglik}.  In addition, the figure shows the
  difference in $g$ from one iteration to the next 
  (dashed line).
  The top left figure, which is for $(m_1,m_2)=(5,4)$, shows the function
  $g=-(2/n)\ell$ converging to its minimum.  The top right figure, for
  $(m_1,m_2)=(6,4)$, shows the function $g$ converging similarly.  In the bottom left figure, for
  $(m_1,m_2)=(7,4)$, the function $g$ diverges to $-\infty$; the dashed line stays at a nonzero negative value.  Finally, in the bottom right figure, for $(m_1,m_2)=(8,4)$, the algorithm converges in one step.  These
  four settings correspond to the cases where the MLE exists uniquely
  (top left), MLEs exist non-uniquely (top and bottom right), and an MLE does not
  exist (bottom left).  We emphasize that while our figure pertains to one particular initialization the observed behavior is similar for other random starting values.
\end{example}

\begin{figure}[t]
\begin{center}
\begin{tabular}{c@{\hspace{1.5cm}}c}
(a) \ \includegraphics[scale=.4]{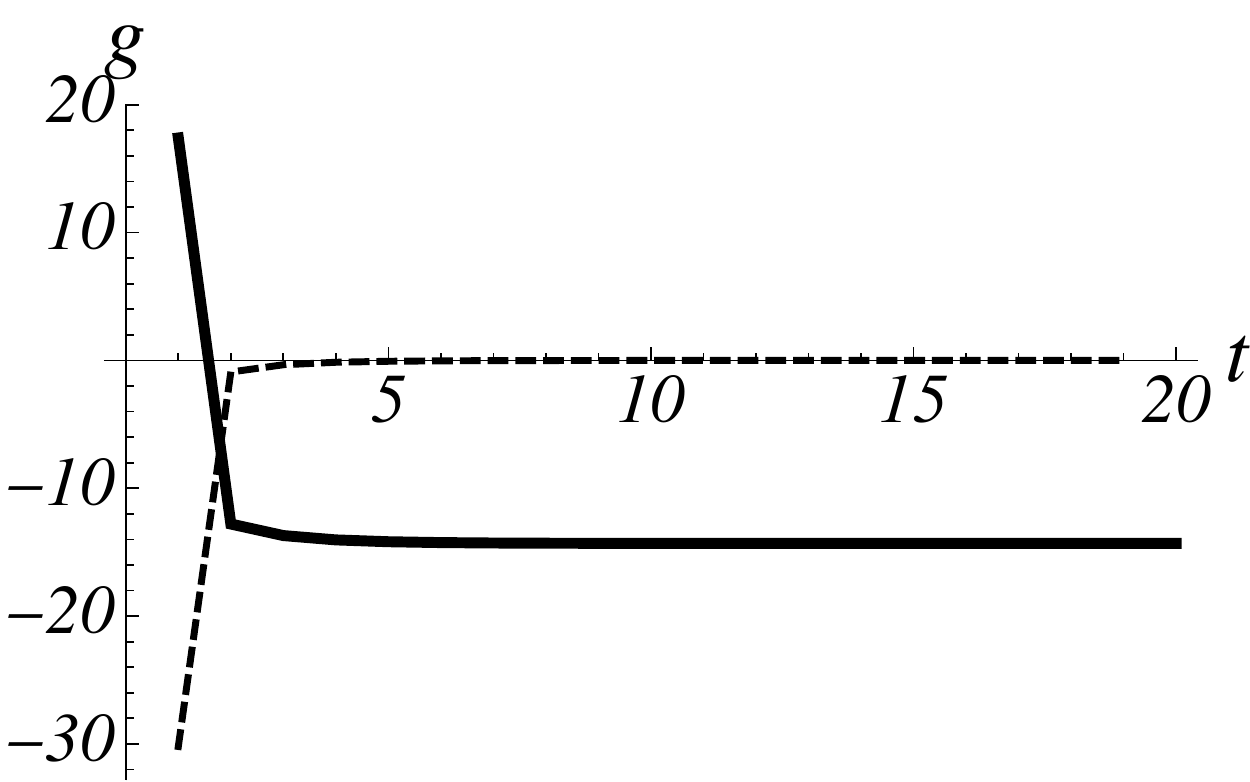} &
(b) \ \includegraphics[scale=.4]{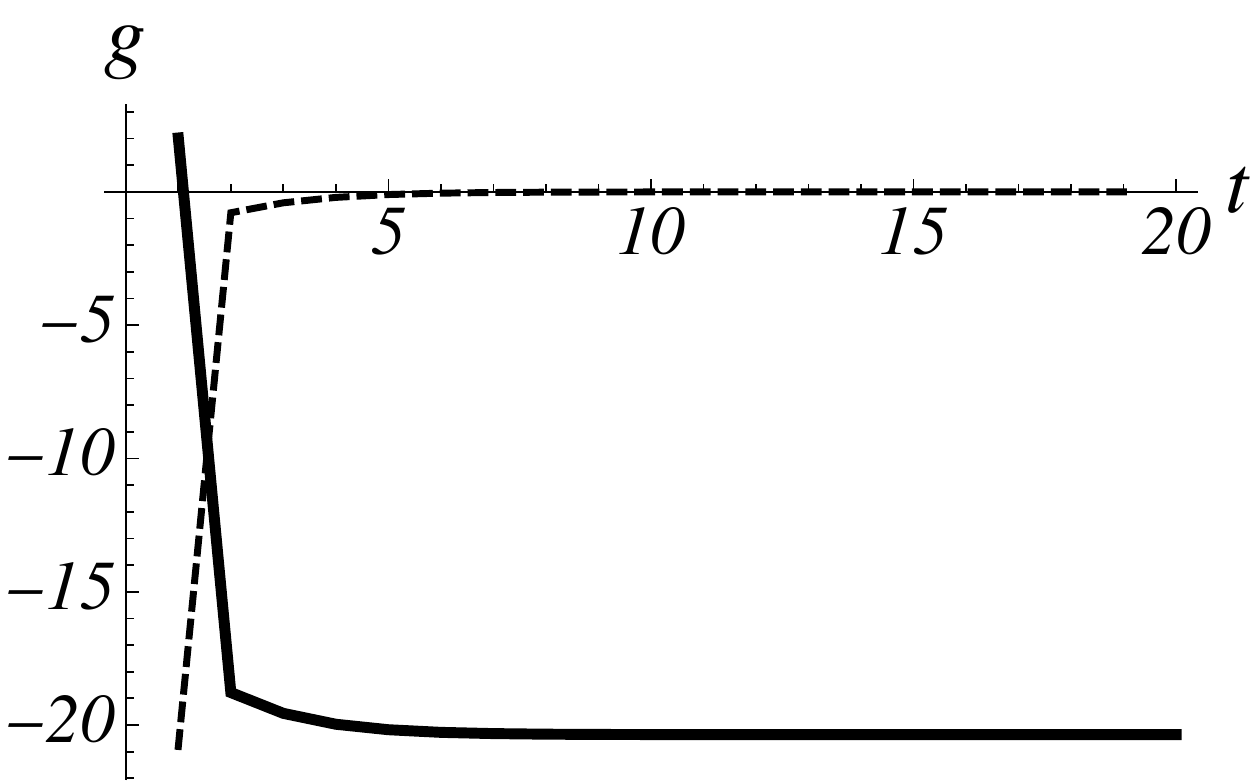}\\[0.1cm]
(c) \ \includegraphics[scale=.4]{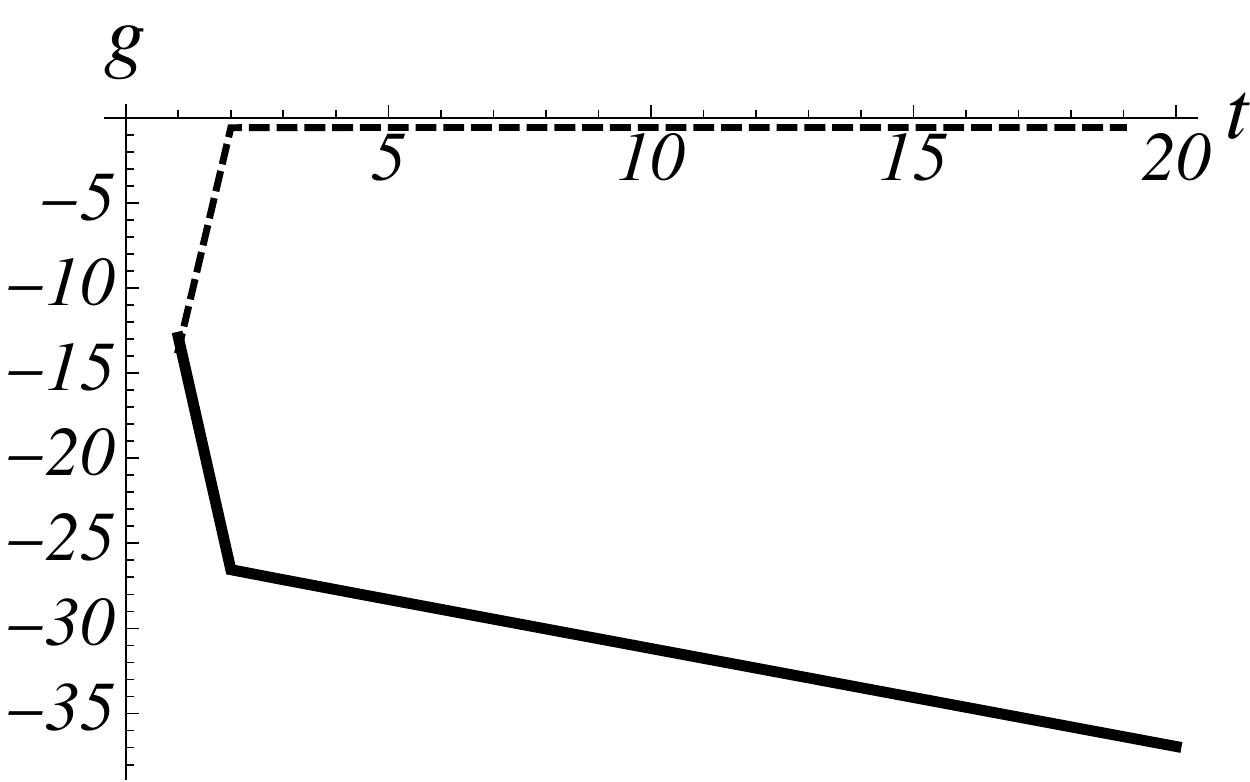} &
(d) \ \includegraphics[scale=.4]{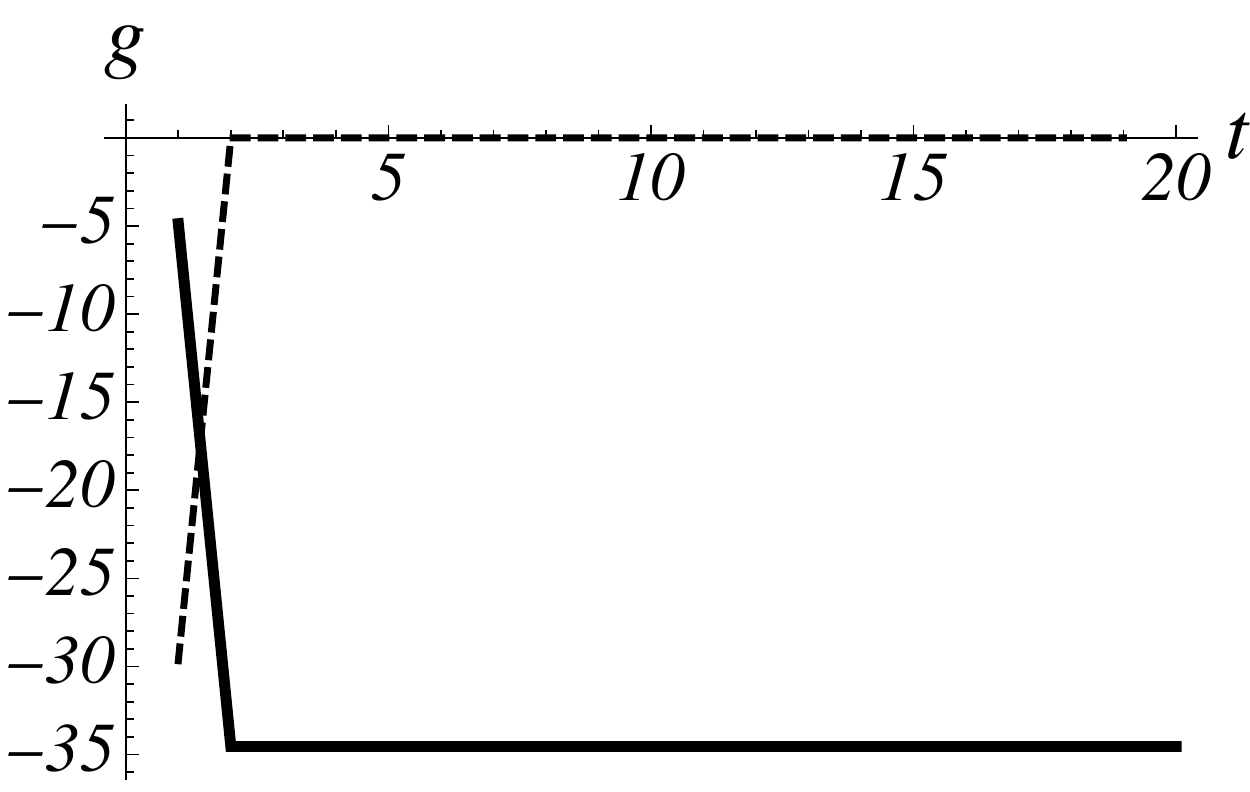}
\end{tabular}
\end{center}
\caption{Plots of profile log-likelihood function $g$  (solid line) and its one-step differences
  (dashed line) against the number of iterations $t$ of the flip-flop algorithm:
  (a) $(m_1,m_2)=(5,4)$ with unique MLE, (b) $(m_1,m_2)=(6,4)$ with the MLE existing non-uniquely, (c) $(m_1,m_2)=(7,4)$ with the MLE non-existing, (d) $(m_1,m_2)=(8,4)$ with the MLE existing non-uniquely.
  }
\label{fig:flipflop}
\end{figure}

As explained in Section~\ref{sec:rectangular-matrices}, the behavior
of the log-likelihood function with respect to existence and
uniqueness of the Kronecker MLE is essentially independent of the
realizations $Y_i$, and merely depends on the triple $(m_1,m_2,n)$.
It is well known that for large enough $n$ the Kronecker MLE exists
uniquely a.s.; this is simply a consequence of the properties
of the usual (vector) normal model.  However, special and at times
somewhat paradoxical properties of the matrix normal model emerge for
small sample size $n$.  Indeed, as we noted in
Example~\ref{ex:flipflop}, for fixed $m_2$ minor differences of $m_1$
may cause substantially different behavior of the log-likelihood
function when the sample size $n$ is small.  To capture this behavior,
we define in this paper three types of sample size thresholds,
which are critical sample sizes at which the a.s.~behavior of the
likelihood function changes.  Our terminology follows \cite{gross:2018,drton:2019}.

\begin{definition}
  \label{def:thresholds}
  (i) The Kronecker MLE \emph{exists} if the
  function $\ell$ achieves its maximum over the domain of definition
  $\mathit{PD}(m_1)\times \mathit{PD}(m_2)$.  It \emph{exists
    uniquely} if in addition all local maxima of $\ell$ have the same
  Kronecker product.

  (ii) We define three positive integer thresholds for the sample
  size.  The \emph{existence threshold} $N_e(m_1,m_2)$ is the integer
  such that the Kronecker MLE exists a.s.\ if and only if
  $n\ge N_e(m_1,m_2)$.  The \emph{uniqueness threshold} $N_u(m_1,m_2)$
  is the integer such that the Kronecker MLE exist uniquely a.s.\ if and only if
  $n\ge N_u(m_1,m_2)$.  Finally, the \emph{boundedness threshold}
  $N_b(m_1,m_2)$ is the integer such that $\ell$ is bounded a.s.\ if and only if
  $n\ge N_b(m_1,m_2)$.
\end{definition}

\begin{remark}
\label{rem:mean-unknown}
  Throughout the paper we assume the expectation of $Y_i$ to be zero and known.
  Standard results 
  yield that
  $N_b(m_1,m_2)+1$, $N_e(m_1,m_2)+1$, and $N_u(m_1,m_2)+1$ are the
  relevant thresholds for the case where the mean matrix in
  $\mathbb{R}^{m_1\times m_2}$ is unknown and also estimated by
  maximum likelihood \cite[Section 3.3]{anderson:2003}.
\end{remark}

The three thresholds from Definition~\ref{def:thresholds} are finite
and no larger than $m_1m_2$.  Specifically, the well-known results on
estimation of an unconstrained Gaussian covariance matrix collected in
\cite[Section 3.2]{anderson:2003} together with Lemma~\ref{lem:unique}
below yield that $n\ge m_1m_2$ is sufficient for a.s.~unique existence
of the Kronecker MLE.  However, this condition is far from necessary.
Indeed, the main theorem of \cite{soloveychik:2016} states that unique
existence holds a.s.~under the much weaker requirement that
\[
  n \;>\; \frac{m_1}{m_2}+\frac{m_2}{m_1}.
\]
To our knowledge, this is the best known sufficient condition for a.s.~(unique) existence.

The condition
\begin{equation}
\label{eq:necessary-cond}
  n\ge \max\left\{\frac{m_1}{m_2},\frac{m_2}{m_1}\right\}
\end{equation}
is necessary for existence; compare also Theorem 1(1) in \cite{soloveychik:2016}.  This is a consequence of the following simple lemma, for which we include a proof in
the appendix.  

\begin{lemma}
  \label{lem:Y-full-row-rank}
  Suppose $m_1\ge m_2$.  If
  $Y=(Y_1,\dots,Y_n)\in\mathbb{R}^{m_1\times n m_2}$ has row rank smaller than
  $m_1$, then $\ell(\Psi_1,\Psi_2)$ is not bounded above on
  $\mathit{PD}(m_1)\times \mathit{PD}(m_2)$.
\end{lemma}

Note from our earlier discussion that the condition from \eqref{eq:necessary-cond} is precisely the requirement needed for
the flip-flop algorithm to have well-defined iterative update steps.  Somewhat confusingly, some of the literature refers to (\ref{eq:necessary-cond})
as the necessary and sufficient condition for existence of the
Kronecker MLE; see, e.g., \cite{dutilleul:1999}.  However,
even when \eqref{eq:necessary-cond} holds, the likelihood
function need not achieve its maximum, or even be bounded (recall Example~\ref{ex:flipflop}).  

In terms of the thresholds we defined, 
the known results from the literature may be summarized as follows.
The floor and ceiling functions are denoted by $\lfloor\cdot\rfloor$ and $\lceil\cdot\rceil$, respectively.

\begin{proposition}
  \label{prop:known-results}
  The three ML thresholds satisfy
  \[
    \max\left\{\frac{m_1}{m_2},\frac{m_2}{m_1}\right\} \;\le\;
    N_b(m_1,m_2)\;\le\; N_e(m_1,m_2)\;\le\; N_u(m_1,m_2) \;\le\; \left\lfloor 
      \frac{m_1}{m_2}+\frac{m_2}{m_1} \right\rfloor + 1.
  \]
\end{proposition}

We remark that $\tfrac{m_1}{m_2}+\tfrac{m_2}{m_1}$ is integer if and only if $m_1=m_2$.
So, for rectangular matrices ($m_1\not=m_2$), the upper bound on
$N_u(m_1,m_2)$ can be written as $\lceil \tfrac{m_1}{m_2}+\tfrac{m_2}{m_1}\rceil$.

\subsection{New contributions}
\label{subsec:intro:new}

Our interest is in precise formulas for the thresholds.
The case where one matrix dimension divides the other is the
simplest.  We derive the following result in the appendix. 
The result considers $m_2\ge 2$.  The case $m_2=1$ reduces to the vector case in which the MLE exists uniquely a.s.~for $n\ge m_1$ and does not exist (with unbounded likelihood) if $n<m_1$ \cite[Section 3.2]{anderson:2003};  so $N_b(m_1,1)=N_e(m_1,1)=N_u(m_1,1)=m_1$.

\begin{theorem}
  \label{thm:divisible}
  Suppose $m_1\ge m_2\ge 2$ and $m_2|m_1$, i.e., $m_1$ is divisible by $m_2$.  Then
  \[
    N_b(m_1,m_2)\;=\;N_e(m_1,m_2)\;=\; \frac{m_1}{m_2}, \qquad
    N_u(m_1,m_2)\;=\;
    \begin{cases}
      3 &\text{if} \ m_1=m_2,\\
      \frac{m_1}{m_2}+1 &\text{if} \ m_1>m_2.
    \end{cases}
  \]
\end{theorem}

For the case where the matrix dimensions do not divide each other,
Proposition~\ref{prop:known-results} 
yields a solution when one matrix dimension is sufficiently large
compared to the other.

\begin{corollary}
  \label{cor:resolved-already}
  Suppose that $m_1> m_2\ge 2$.  Let $r =m_1\bmod m_2$ be the remainder in
  integer division, so $r\in\{0,\dots,m_2-1\}$.  If $r\ge 1$ and
  \[
    \left\lfloor \frac{m_1}{m_2}\right\rfloor \;> 
     1+ \frac{r^2}{m_2(m_2-r)}, 
  \]
  then 
  \[
    N_b(m_1,m_2)\;=\; N_e(m_1,m_2)\;=\; N_u(m_1,m_2)\;=\; \left\lfloor
      \frac{m_1}{m_2}\right\rfloor +1.
  \]
  This formula holds, in particular, when $m_1\ge m_2^2$ and $r\ge 1$.
\end{corollary}

Our new results in Theorem~\ref{thm:divisible} and Corollary \ref{cor:resolved-already} can be roughly interpreted as saying that, 
as long as $m_1$ is divisible or approximately divisible by $m_2$, 
then the sample size requirements for existence of an MLE are the same as the sample size requirements for 
the steps of the 
flip-flop algorithm to be well defined. This equivalence is perhaps the reason 
why in the early literature on Kronecker MLEs, these two sample sizes were 
conflated. However, as illustrated in Example~\ref{ex:flipflop}, 
the 
flip-flop algorithm can diverge even though each step is well-defined. 
The main additional results in our paper are to understand and describe 
this discrepancy as well as the difference between existence and unique existence of the MLE. Specifically, we provide 
formulas for the three thresholds we defined
in the regime where the upper bound from
Proposition~\ref{prop:known-results} leaves the largest gap, namely,
the case where
\[
  2m_2 \;\ge\; m_1 \;\ge\; m_2.
\]
We state our main theorem here.

\begin{theorem}
  \label{thm:main}
  Suppose that $2m_2\ge m_1\ge m_2$.  
  Then 
\[
 N_u(m_1,m_2) = \begin{cases}
 3 & \mbox{if $m_1=m_2$}, \\
 2 & \mbox{if $m_1=m_2+1$}, \\
 3 & \mbox{otherwise}, \end{cases}
\]
and
\[
 N_e(m_1,m_2) = N_b(m_1,m_2) = \begin{cases}
 1 & \mbox{if $m_1=m_2$}, \\
 2 & \mbox{if $m_1>m_2$ and $m_1-m_2 \vert m_2$}, \\
 3 & \mbox{otherwise}. \end{cases}
\]
\end{theorem}

The ingredients needed to establish Theorem~\ref{thm:main} will be
developed in the remainder of the paper.  How they fit together is
also outlined in the proof of Theorem~\ref{thm:main} that we include
in the appendix.  When combined with additional calculations using
Gr\"obner basis methods to check algebraic conditions given in
Theorems~\ref{thm:necessary-rank} and \ref{thm:sufficient-rank},
Theorem~\ref{thm:main} provides $N_u(m_1,m_2)$, $N_e(m_1,m_2)$, and
$N_b(m_1,m_2)$ for small $m_1$ and $m_2$; see Table
\ref{table:threshold}.

\medskip
The remainder of the paper is organized as follows.  We begin by recalling preliminary results
concerning convexity properties of the negated log-likelihood
function, and we introduce a profile likelihood function
(Section~\ref{sec:preliminaries}).  We then give algebraic conditions
for existence and uniqueness of the Kronecker MLE
(Section~\ref{sec:algebr-cond-exist}).  These are formulated in terms
of rank drops, meaning the extent to which the rank of a data matrix
may be reduced through certain linear transformations.  The sufficient
condition also appears in a more geometric form in the proofs in
\cite{soloveychik:2016}.  The key ingredients for the proof of
Theorem~\ref{thm:divisible} are derived in
Section~\ref{sec:square-matrices}.  The proof of
Theorem~\ref{thm:main} requires a study of the case of sample size
$n=2$ and is developed in Section~\ref{sec:rectangular-matrices}.  Our
arguments use invariance of the likelihood surface under group actions,
and we use certain canonical forms for generic data matrices under the
group action.  For data in such canonical form, we are able to explicitly give the
critical points of the likelihood function (Section~\ref{sec:mle}).
We end with a brief conclusion
(Section~\ref{sec:conclusion}).

\begin{table}[t]
\caption{ML thresholds: 
Unique existence $N_u(m_1,m_2)$ (left); \hspace{\textwidth}
Bounded likelihood/existence $N_e(m_1,m_2)=N_b(m_1,m_2)$ (right).}
\label{table:threshold}
\begin{small}
\begin{minipage}[h]{0.43\textwidth}
    \begin{tabular}{c|*{8}{c@{\;\;\;\,}}c@{\;\;\,}c}
\hline
      \small\diagbox[width=7ex]{$m_1$}{$m_2$}
      &1 & 2& 3& 4& 5& 6& 7& 8& 9& 10 \\
      \hline
      1&1\\
      2&2& 3 \\
      3&3& 2& 3 \\
      4&4& 3& 2& 3 \\
      5&5& 3& 3& 2& 3 \\
      6&6& 4& 3& 3& 2& 3 \\
      7&7& 4& 3& 3& 3& 2& 3 \\
      8&8& 5& 3 
             & 3& 3& 3& 2 &3 \\
      9&9& 5& 4& 3& 3& 3& 3 &2 &3 \\
     10&10& 6& 4& 3& 3& 3& 3 &3 &2 &3 \\
\hline
    \end{tabular}
\end{minipage}
\qquad
\begin{minipage}[h]{0.43\textwidth}
    \begin{tabular}{c|*{8}{c@{\;\;\;\,}}c@{\;\;\,}c}
\hline
      \small\diagbox[width=7ex]{$m_1$}{$m_2$}
      &1 & 2& 3& 4& 5& 6& 7& 8& 9& 10 \\
      \hline
      1&1\\
      2&2& 1 \\
      3&3& 2& 1 \\
      4&4& 2& 2& 1 \\
      5&5& 3& 3& 2& 1 \\
      6&6& 3& 2& 2& 2& 1 \\
      7&7& 4& 3& 3& 3& 2& 1 \\
      8&8& 4& 3& 2& 3& 2& 2& 1 \\
      9&9& 5& 3& 3& 3& 2& 3& 2& 1 \\
     10&10& 5& 4& 3& 2& 3& 3& 2& 2& 1 \\
\hline
    \end{tabular}
\end{minipage}
\end{small}
\end{table}

\section{Preliminaries}
\label{sec:preliminaries}

\subsection{Geodesic convexity and uniqueness of MLE}

The log-likelihood function $\ell$ defined
in~(\ref{eq:log-likelihood}) is not concave.  However, it can be shown
to be g-concave, that is, $\ell$ is concave along suitable geodesics
between any pair of matrices in
$\mathit{PD}(m_1)\times\mathit{PD}(m_2)$.  The geodesics are obtained
from the geodesics in $\mathit{PD}(m_1)$ and in $\mathit{PD}(m_2)$, which take the form
\begin{equation}
  \label{eq:geodesic}
  \gamma_t^{(j)}(Q_0,Q_1) \;=\;Q_0^{\frac{1}{2}}\left(
    Q_0^{-\frac{1}{2}}Q_1Q_0^{-\frac{1}{2}}\right)^t
  Q_0^{\frac{1}{2}}, \quad t\in[0,1],
\end{equation}
when linking two matrices $Q_0,Q_1\in\mathit{PD}(m_j)$, $j=1,2$. The g-concavity
of $\ell$ amounts to 
\[
\ell\bigl(\gamma_t^{(1)}\bigl(Q_{0}^{(1)},Q_{1}^{(1)}\bigr), \gamma_t^{(2)}\bigl(Q_{0}^{(2)},Q_{1}^{(2)}\bigr)\bigr) \ge 
t\cdot\ell\bigl(Q_{0}^{(1)},Q_{0}^{(2)}\bigr) +
(1-t)\cdot\ell\bigl(Q_{1}^{(1)},Q_{1}^{(2)}\bigr)
\]
for all $t\in[0,1]$, $Q_0^{(j)},Q_{1}^{(j)}\in\mathit{PD}(m_j)$, $j=1,2$.  This property of $\ell$ was observed in \cite{wiesel:2012} and yields the following
fact; see also \cite[Chap.~6]{rapcsak:1997}.

\begin{lemma}
  \label{lem:gconvex}
  Every critical point (i.e., point of zero gradient) and, in particular, every local maximum of
  $\ell$ is a global maximum.
\end{lemma}

Suppose $\ell$ has two distinct maxima.  Then, by concavity, $\ell$ is
constant along the geodesic linking them.  We may extend the geodesic
by considering $t\in\mathbb{R}$ in~(\ref{eq:geodesic}).  Along this
extended geodesic, $\ell$ remains constant even as the underlying
matrices diverge or approach the boundary of $\mathit{PD}(m_1)$ or
$\mathit{PD}(m_2)$.  This can be used to guarantee uniqueness of the
Kronecker MLE as noted in \cite{soloveychik:2016}.  Call the
log-likelihood function $\ell$ \emph{coercive} if
\[
  \lim_{t\to\infty} \ell\bigl(\Psi_1^{(t)},\Psi_2^{(t)}\bigr)
  \;=\;-\infty
\]
for all sequences $\bigl(\Psi_1^{(t)},\Psi_2^{(t)}\bigr)$ that diverge or
approach the boundary of $\mathit{PD}(m_1)\times\mathit{PD}(m_2)$.  If
$\ell$ is coercive, then clearly the Kronecker MLE exists and it
exists uniquely based on our above discussion.  However, more is true according to the following lemma that is  proven as part of Lemma 4 in \cite{soloveychik:2016}.

\begin{lemma}
  \label{lem:unique}
  The Kronecker MLE exists uniquely if and only if the log-likelihood
  function $\ell$ is coercive.
\end{lemma}

In the following we will often consider properties that the normal data matrices $Y_1,\dots,Y_n$ possess almost surely.  More precisely, the properties will hold as long as the data lie outside certain, not further specified lower-dimensional sets that are defined by polynomial equations.  We indicate this fact by speaking of data matrices that are \emph{generic}.

\subsection{Profile likelihood}
\label{sec:profile}

For any $\Psi_2\in\mathit{PD}(m_2)$, the log-likelihood function
admits a section
$\ell_{\Psi_2}: \mathit{PD}(m_1)\to \mathbb{R}$ given by
$\ell_{\Psi_2}(\Psi_1)= \ell(\Psi_1,\Psi_2)$.

\begin{lemma}
  \label{lem:profile}
  Suppose $n m_2\ge m_1\ge m_2$.  For generic data $Y_1,\dots,Y_n$,
  every section $\ell_{\Psi_2}$, $\Psi_2\in\mathit{PD}(m_2)$,
  achieves its maximum on $\mathit{PD}(m_1)$ uniquely at
  \[
  \Psi_1(\Psi_2) = \left(\frac{1}{nm_2} \sum_{i=1}^n Y_i\Psi_2 Y_i^T\right)^{-1}.
  \]
\end{lemma}
\begin{proof}
  According to the well-known results for ML estimation of a Gaussian
  covariance matrix \cite[Section 3.2]{anderson:2003}, the claim is
  true for a particular restriction $\ell_{\Psi_2}$ if the positive
  semi-definite $m_1\times m_1$ matrix
  \[
  \sum_{i=1}^n Y_i\Psi_2 Y_i^T
  \]
  is non-singular.  The matrix is a sum of positive semi-definite
  matrices.  Hence, a vector $v\in\mathbb{R}^{m_2}$ is in its kernel
  if and only if
  $v^T Y_i\Psi_2Y_i^T v=0$
  for all $i=1,\dots,n$.  Since $\Psi_2$ is positive definite, this
  holds if and only if $v$ is in the kernel of each matrix $Y_i^T$ if
  and only if $v$ is in the kernel of
  \[
  \sum_{i=1}^n Y_i Y_i^T.
  \]
  Being a sum of $nm_2\ge m_1$ generic rank 1 matrices, this
  $m_1\times m_1$ matrix has full rank $m_1$.  The kernel is thus zero.
\end{proof}

In the regime of interest, with $nm_2\ge m_1\ge m_2$,
Lemma~\ref{lem:profile} ensures that for generic data $Y_1,\dots,Y_n$
the profile log-likelihood function
\[
  \ell_\text{prof}:\Psi_2\mapsto\ell(\Psi_1(\Psi_2),\Psi_2)
\]
is well-defined on $\mathit{PD}(m_2)$.  Now, $\ell$ is bounded
above/achieves its maximum on
$\mathit{PD}(m_1)\times \mathit{PD}(m_2)$ if and only if
$\ell_\text{prof}$ is bounded above/achieves its maximum on
$\mathit{PD}(m_2)$.  This in turn is equivalent to the function
\begin{equation}
  \label{eq:profile-neglik}
  g(\Psi) = m_2\log\det\left(\sum_{i=1}^n Y_i\Psi Y_i^T\right) - m_1\log\det(\Psi)
\end{equation}
being bounded below/achieving its minimum on $\mathit{PD}(m_2)$.  We
note that
\begin{align}
  \label{eq:profile-neglik-Y}
  g(\Psi) 
  &= m_2\log\det\left(Y[I_n \otimes \Psi] Y^T\right) - m_1\log\det(\Psi),
\end{align}
where
$Y = \begin{pmatrix} Y_1, \dots, Y_n \end{pmatrix}
\in\mathbb{R}^{m_1\times nm_2}$.  The function $g$ is geodesically
convex (g-convex).  This follows from its definition as negatived
profile of a g-concave log-likelihood function.  It can also be seen
directly by observing that $\log\det(\cdot)$ is linear along the
geodesics from~(\ref{eq:geodesic}) and verifying that the first term
in~(\ref{eq:profile-neglik}) is g-convex; see Lemma 2
in \cite{wiesel:2012}.

Call $g$ coercive if $g\bigl(\Psi^{(t)}\bigr)$ tends to $+\infty$ for any
sequence $\Psi^{(t)}$ that diverges or approaches the boundary of
$\mathit{PD}(m_2)$.  Then our observations may be summarized as follows.

\begin{lemma}
  \label{lem:profile-likelihood}
  The Kronecker MLE exists if and only if the function $g$
  from~(\ref{eq:profile-neglik}) achieves its minimum on
  $\mathit{PD}(m_2)$.  The Kronecker MLE exists uniquely if and only
  if $g$ is coercive.
\end{lemma}
\begin{proof}
  The first assertion of this lemma is clear from the definition of
  $g$, and the second claim follows from Lemma~\ref{lem:unique}.
\end{proof}

\subsection{Group action}
\label{sec:group-action}

An important ingredient to our later analysis is the fact that a group
action allows one to consider data in canonical position.  Let
$\mathrm{GL}(m,\mathbb{R})$ be the general linear group of $m\times m$
real invertible matrices.  The direct product
$\mathrm{GL}(m_1,\mathbb{R})\times\mathrm{GL}(m_2,\mathbb{R})$ acts
naturally on a data set comprised of matrices
$Y_1,\dots,Y_n\in\mathbb{R}^{m_1\times m_2}$.  For
$A\in\mathrm{GL}(m_1,\mathbb{R})$ and
$B\in\mathrm{GL}(m_2,\mathbb{R})$, the action is
\begin{align*}
  Y_i \mapsto AY_iB, \quad i=1,\dots,n.
\end{align*}
Now two data sets $(Y_1,\dots,Y_n)$ and $(Y_1',\dots,Y_n')$ are in the same orbit under the group action if one can be transformed into the other using a pair $(A,B)\in\mathrm{GL}(m_1,\mathbb{R})\times\mathrm{GL}(m_2,\mathbb{R})$.

Recall that the log-likelihood surface of a model is the graph of its
log-likelihood function.

\begin{lemma}
  \label{lem:group-action}
  If two data sets are in the same
  $\mathrm{GL}(m_1,\mathbb{R})\times\mathrm{GL}(m_2,\mathbb{R})$-orbit,
  then their log-likelihood surfaces are translations of one another.
\end{lemma}
\begin{proof}
  Let $\ell$ be the log-likelihood function
  from~(\ref{eq:log-likelihood}) for data
  $Y_1,\dots,Y_n\in\mathbb{R}^{m_1\times m_2}$. Let
  $A\in\mathrm{GL}(m_1,\mathbb{R})$ and
  $B\in\mathrm{GL}(m_2,\mathbb{R})$.  Define $Y_i'=AY_iB$ for
  $i=1,\dots,n$.  The log-likelihood function for the data
  $(Y_1',\dots,Y_n')$ is
  \begin{align*}
    \ell'(\Psi_1,\Psi_2)
    &= n m_2\log\det(\Psi_1) + n m_1\log\det(\Psi_2) -
      \tr\bigg(\Psi_1A \sum_{i=1}^n
      Y_iB\Psi_2B^TY_i^TA^T\bigg)\\
    &= \ell\bigl(A^T\Psi_1A,B\Psi_2B^T\bigr) + c,
  \end{align*}
  where $c$ is a constant that depends on $n$, $m_1$, $m_2$, $\det(A)$
  and $\det(B)$.  The maps $\Psi_1\mapsto A^T\Psi_1A$ and
  $\Psi_2\mapsto B\Psi_2 B^T$ are bijections from $\mathit{PD}(m_1)$
  to $\mathit{PD}(m_1)$ and $\mathit{PD}(m_2)$ to $\mathit{PD}(m_2)$,
  respectively.  Therefore, subtracting $c$ from each function value
  translates the graph of $\ell'$ into the graph of $\ell$.
\end{proof}

\section{Algebraic conditions for existence and uniqueness}
\label{sec:algebr-cond-exist}

In this section we prove a necessary condition for existence of the
Kronecker MLE as well as a sufficient condition for its unique
existence.  Both conditions involve the rank of 
\begin{equation}
  \label{eq:def-Y}
  Y=(Y_1,\dots,Y_n)\in\mathbb{R}^{m_1\times n m_2}
\end{equation}
after linear transformation using certain
$nm_2\times nm_2$ matrices in Kronecker product form.
  
\subsection{Necessary condition for existence}
\label{sec:necessary-rank}

The following condition strengthens the necessary condition from
Lemma~\ref{lem:Y-full-row-rank}.  
\begin{theorem}
  \label{thm:necessary-rank}
  (i) If there exists a matrix $X\in\mathbb{R}^{m_2\times m_2}$ such
  that
  \[
    \rk
    \begin{pmatrix}
      Y_1X,\dots,Y_nX
    \end{pmatrix}
    <\frac{m_1}{m_2}\rk(X),
  \]
  then the log-likelihood function $\ell$ is unbounded and
  the Kronecker MLE does not exist.\\[0.1cm]
  (ii) If there exists a non-zero and singular matrix
  $X\in\mathbb{R}^{m_2\times m_2}$ such that
  \[
    \rk
    \begin{pmatrix}
      Y_1X,\dots,Y_nX
    \end{pmatrix}
    \le\frac{m_1}{m_2}\rk(X),
  \]
  then the log-likelihood function $\ell$ is not coercive and the
  Kronecker MLE does not exist uniquely.
\end{theorem}
\begin{proof}
  (i) Let $X\in\mathbb{R}^{m_2\times m_2}$ satisfy the assumed rank
  condition.  If $X$ is invertible, then the data matrix $Y$
  from~(\ref{eq:def-Y}) has $\rk(Y)<m_1$ and the likelihood function
  is unbounded by Lemma~\ref{lem:Y-full-row-rank}.  Hence, only the
  case where $\rk(X)=k<m_1$ needs to be considered.
  
  Let $q_1,\dots,q_{m_2}$ be the eigenvectors of $XX^T$, with
  corresponding eigenvalues $d_1,\dots,d_{m_2}$.  Since
  $\rk(XX^T)=\rk(X)=k$, we may assume that
  $d_j=0$ for $j\ge k+1$, in which case
  \begin{align*}
    XX^T &= \sum_{j=1}^k d_j q_jq_j^T.
  \end{align*}
  Define the positive semidefinite matrix
  \begin{align*}
    \Psi &= \sum_{j=k+1}^{m_2}  q_jq_j^T.
  \end{align*}
  Then $XX^T+\Psi$ is positive definite.
  We claim that 
  \begin{equation}
    \label{eq:lim-infinity}
    \lim_{\lambda\to\infty} g(\lambda XX^T+\Psi) \;=\; -\infty,
  \end{equation}
  which implies the theorem's assertion via
  Lemma~\ref{lem:profile-likelihood}.

  Let $r$ be the rank of 
  \[
    (Y_1X,\dots,Y_nX) \;=\; Y (I_n \otimes X).
  \]
  Then $r$ is also the rank of 
  \[
    Y (I_n \otimes X)(I_n \otimes X)^T Y^T \;=\;
    Y (I_n \otimes XX^T) Y^T.
  \]
  By Lemma~\ref{lem:degree-det} below, the determinant of 
  \begin{align*}
    Y[I_n\otimes (\lambda XX^T+\Psi)]Y^T &= 
    \lambda\cdot Y[I_n\otimes XX^T]Y^T +  Y[I_n\otimes\Psi]Y^T
  \end{align*}
  is a polynomial of degree $r$ in $\lambda$.
  Lemma~\ref{lem:degree-det} also yields that
  $\det(\lambda XX^T+\Psi)$ is a polynomial of degree $k$ in
  $\lambda$.  By assumption $rm_2<km_1$.  Therefore, 
  \begin{equation}
    \label{eq:ratio-sequence}
    \lim_{\lambda\to\infty} e^{g(\lambda XX^T+\Psi)}=
    \lim_{\lambda\to\infty} \frac{\left\{\det\left(Y \left[I_n
          \otimes (\lambda XX^T+\Psi)\right] Y^T\right)\right\}^{m_2}}{\left\{\det(\lambda
        XX^T+\Psi)\right\}^{m_1}} = 0.
  \end{equation}
  Taking the logarithm yields the claim from~(\ref{eq:lim-infinity}).\\[0.1cm]
  (ii) Proceeding as in case (i) we obtain
  in~(\ref{eq:ratio-sequence}) a ratio of two polynomials of equal
  degree and a finite and positive limit.  It follows that $g$
  converges to a finite limit and, thus, is not coercive.  An
  application of Lemma~\ref{lem:profile-likelihood} yields the claim.
\end{proof}

\begin{lemma}
  \label{lem:degree-det}
  Let $A,B\in\mathbb{R}^{m\times m}$ be two positive semidefinite
  matrices whose sum $A+B$ is positive definite.  Let $\rk(A)=r$.
  Then $\det(\gamma A+B)$ is a degree $r$ polynomial in $\gamma$.
\end{lemma}
\begin{proof}
  Choose an invertible matrix $C$ such
  that $C^TC=A+B$.  Let $Q^TDQ$ be the spectral decomposition of
  $C^{-T}AC^{-1}$ with $D=\diag(d_1,\dots,d_r,0,\dots,0)$ and $d_j>0$
  for $j\le r$.  Then
  \begin{align*}
    \det(\gamma A+B)&=\det\big((\gamma-1)A+(A+B)\big)\\
                    &= \det(C)^2
                      \det\big((\gamma-1)D+I\big) = \det(A+B)^2\prod_{j=1}^r (d_j\gamma
                      +1-d_j).
  \end{align*}
\end{proof}

\begin{remark}
  The eigenvalues $d_j$ in the above proof are also the eigenvalues of
  $(A+B)^{-1}A$.  If $v\in\mathrm{ker}(B)$ then 
  \[
    (A+B)^{-1}Av = (A+B)^{-1}(A+B)v = v,
  \]
  so that $v$ is an eigenvector for eigenvalue 1.
\end{remark}

\subsection{Sufficient conditions for existence and uniqueness}
\label{sec:sufficient-rank}

\begin{theorem}
  \label{thm:sufficient-rank}
  Let $\rk(Y_1,\dots,Y_n)=m_1$.  
  \\
  (i) If all singular matrices $X\in\mathbb{R}^{m_2\times m_2}$ satisfy
  \[
    \rk
    \begin{pmatrix}
      Y_1X,\dots,Y_nX
    \end{pmatrix}
    \ge \frac{m_1}{m_2}\rk(X),
  \]
  then the log-likelihood function $\ell$ is bounded from above.\\[0.1cm]
  (ii)  If all non-zero singular matrices $X\in\mathbb{R}^{m_2\times
    m_2}$ satisfy 
  \[
    \rk
    \begin{pmatrix}
      Y_1X,\dots,Y_nX
    \end{pmatrix}
    >\frac{m_1}{m_2}\rk(X),
  \]
  then the log-likelihood function $\ell$ is coercive and the
  Kronecker MLE exists uniquely.
\end{theorem}

\begin{proof}
(i) Assume that the log-likelihood function $\ell$ is not bounded from
above.  Then there exists a sequence $\Psi^{(t)}$, $t=1,2,\ldots$, in $\mathit{PD}(m_2)$ such that $g\bigl(\Psi^{(t)}\bigr)\to -\infty$ as $t\to\infty$.
Let $\Psi^{(t)}=Q^{(t)} \Lambda^{(t)} \bigl(Q^{(t)}\bigr)^T$ be the spectral decomposition.
The set of orthogonal matrices is compact and, passing to a subsequence
if necessary,
we may assume the sequence $Q^{(t)}$ to be convergent.
By Lemma~\ref{lem:classify} in the appendix, again passing to a
subsequence if necessary, we may assume the diagonal elements of
$\Lambda^{(t)}$ to be such that  the resulting sequence $\Psi^{(t)}$ is of the form
\begin{equation}
\label{eq:Psi}
 \Psi^{(t)} = \epsilon_1^{(t)} \Psi_1^{(t)} + \epsilon_2^{(t)} \Psi_2^{(t)} + \cdots + \epsilon_K^{(t)} \Psi_K^{(t)}, \quad K\le m_2,
\end{equation}
where $\Psi_k^{(t)}$ is a sequence of positive semidefinite matrices
that converges to a limit 
$\Psi_k$ and such that for all $t$ it holds that $\rk\Psi_k^{(t)}=\rk\Psi_k$ and
\begin{equation}
\label{eq:directsum}
 \mathrm{Im}\,\Psi_1^{(t)}\oplus\cdots\oplus\mathrm{Im}\,\Psi_K^{(t)}=
 \mathbb{R}^{m_2}.
\end{equation}
Moreover, the scalars in \eqref{eq:Psi} are positive,
$\epsilon_k^{(t)}>0$, and satisfy $\epsilon_{k+1}^{(t)}/\epsilon_{k}^{(t)}\to 0$.
Note also that in~\eqref{eq:Psi} we must have $K\ge 2$ because if $K=1$, then $\lim_{t\to\infty}\Psi_1^{(t)}=\Psi_1\in \mathit{PD}(m_2)$ and $\lim_{t\to\infty} g\bigl(\Psi^{(t)}\bigr) = \lim_{t\to\infty} g\bigl(\epsilon_1^{(t)}\Psi_1^{(t)}\bigr) = \lim_{t\to\infty} g\bigl(\Psi_1^{(t)}\bigr) = \lim_{t\to\infty} g(\Psi_1)>-\infty$.

Next, we redefine $\Psi_k^{(t)}$ and $\epsilon_k^{(t)}$ as
\begin{align*}
 & \epsilon_k^{(t)} := \epsilon_k^{(t)}-\epsilon_{k+1}^{(t)},\quad k=1,\ldots, K-1, \quad \epsilon_K^{(t)} := \epsilon_K^{(t)}, \\
 & \Psi_k^{(t)} := \Psi_1^{(t)} + \cdots + \Psi_k^{(t)},\quad
   \Psi_k := \Psi_1 + \cdots + \Psi_k,\quad k=1,\ldots,K.
\end{align*}
The new $\Psi^{(t)}$ remains of the form (\ref{eq:Psi}) with
$\Psi_k^{(t)}\to\Psi_k$, $\rk\Psi_k^{(t)}=\rk\Psi_k$.  Similarly, the new
$\epsilon_k^{(t)}$ remains positive with $\epsilon_{k+1}^{(t)}/\epsilon_{k}^{(t)}\to 0$.
However, instead of (\ref{eq:directsum}), we now have
\begin{equation}
\label{eq:monotone}
 \mathrm{Im}\,\Psi_1^{(t)}\subset\cdots\subset\mathrm{Im}\,\Psi_K^{(t)}=\mathbb{R}^{m_2}.
\end{equation}
Write $r_k=\rk\Psi_k$.
Then, by (\ref{eq:monotone}) and Lemma~\ref{lem:degree-det},
\begin{align*}
 \det\bigl(\Psi^{(t)}\bigr)
=& \det\bigl(\epsilon_1^{(t)}\Psi_1^{(t)} + \epsilon_2^{(t)} \Psi_2^{(t)} + \cdots + \epsilon_K^{(t)} \Psi_K^{(t)}\bigr) \\
=& \bigl(\epsilon_K^{(t)}\bigr)^{m_2} \det\bigl(\bigl(\epsilon_1^{(t)}/\epsilon_K^{(t)}\bigr)\Psi_1^{(t)} + \bigl(\epsilon_2^{(t)}/\epsilon_K^{(t)}\bigr)\Psi_2^{(t)} + \cdots + \Psi_K^{(t)}\bigr) \\
\asymp& \bigl(\epsilon_K^{(t)}\bigr)^{m_2} \bigl(\epsilon_1^{(t)}/\epsilon_K^{(t)}\bigr)^{r_1}
 \bigl(\epsilon_2^{(t)}/\epsilon_K^{(t)}\bigr)^{r_2-r_1} \cdots
 \bigl(\epsilon_{K-1}^{(t)}/\epsilon_K^{(t)}\bigr)^{r_{K-1}-r_{K-2}} \\
=& \bigl(\epsilon_K^{(t)}\bigr)^{m_2} \bigl(\epsilon_1^{(t)}/\epsilon_2^{(t)}\bigr)^{r_1}
 \bigl(\epsilon_2^{(t)}/\epsilon_3^{(t)}\bigr)^{r_2} \cdots
 \bigl(\epsilon_{K-1}^{(t)}/\epsilon_K^{(t)}\bigr)^{r_{K-1}} \\
=& \bigl(\epsilon_K^{(t)}\bigr)^{m_2} \bigl(\gamma_1^{(t)}\bigr)^{r_1} \bigl(\gamma_2^{(t)}\bigr)^{r_2} \cdots \bigl(\gamma_{K-1}^{(t)}\bigr)^{r_{K-1}},
\end{align*}
where $\gamma_k^{(t)}=\epsilon_k^{(t)}/\epsilon_{k+1}^{(t)}$.
Note that $\gamma_k^{(t)}\to\infty$ as $t\to\infty$.

Let $M^{(t)}=\sum_{i=1}^n Y_i\Psi^{(t)} Y_i^T$ and $M_k^{(t)}=\sum_{i=1}^n Y_i\Psi_k^{(t)} Y_i^T$.
Then, as $t\to\infty$, $M_k^{(t)}\to M_k=\sum_{i=1}^n Y_i\Psi_k Y_i^T$.
The monotonicity property (\ref{eq:monotone}) is inherited as
\[
 \mathrm{Im}\,M_1^{(t)}\subset\cdots\subset\mathrm{Im}\,M_K^{(t)}=\mathbb{R}^{m_1},
\]
and therefore
\[
 \det\bigl(M^{(t)}\bigr) \asymp \bigl(\epsilon_K^{(t)}\bigr)^{m_2} \bigl(\gamma_1^{(t)}\bigr)^{\rk(M_1^{(t)})} \bigl(\gamma_2^{(t)}\bigr)^{\rk(M_2^{(t)})} \cdots \bigl(\gamma_{K-1}^{(t)}\bigr)^{\rk(M_{K-1}^{(t)})}.
\]
Define
\[
 R(k) := \min_{X\in\mathbb{R}^{m_2\times m_2},\,\rk X=k}\rk (Y_1 X,\ldots,Y_n X).
\]
The assumption in part (i) of the theorem is that
\begin{equation}
\label{eq:cond1}
 R(r)-(m_1/m_2) r \ge 0, \quad r=1,\ldots,m_2-1.
\end{equation}
Then, the order of $\det(M^{(t)})$ is bounded from below by
\[
 \bigl(\epsilon_K^{(t)}\bigr)^{m_1} \bigl(\gamma_0^{(t)}\bigr)^{R(r_0)} \bigl(\gamma_1^{(t)}\bigr)^{R(r_1)} \cdots \bigl(\gamma_{K-1}^{(t)}\bigr)^{R(r_{K-1})},
\]
and the order of
\[
 e^{g(\Psi^{(t)})} = \frac{\det\bigl(M^{(t)}\bigr)^{m_2}}{\det\bigl(\Psi^{(t)}\bigr)^{m_1}}
\]
is bounded from below by
\begin{equation}
\label{eq:below}
 \bigl(\gamma_0^{(t)}\bigr)^{m_2 R(r_0)-m_1 r_0} \bigl(\gamma_1^{(t)}\bigr)^{m_2 R(r_1)-m_1 r_1} \cdots \bigl(\gamma_{K-1}^{(t)}\bigr)^{m_2 R(r_{K-1})-m_1 r_{K-1}}.
\end{equation}
Under the condition (\ref{eq:cond1}), the product in (\ref{eq:below})
and, thus, also $e^{g(\Psi^{(t)})}$ does not converge to 0.
This means that $g\bigl(\Psi^{(t)}\bigr)$ is bounded from below and cannot
diverge to $-\infty$.
This is a contradiction.

\medskip
(ii) The assumption is now that 
\begin{equation}
\label{eq:cond2}
 R(r)-(m_1/m_2) r > 0, \quad r=1,\ldots,m_2-1.
\end{equation}
Let $\Psi^{(t)}$ be a sequence in $\mathit{PD}(m_2)$ such that
$\Psi^{(t)}\to\Psi^0$, where $\Psi^0$ is singular.  Assume that the
likelihood function $\ell$ is not coercive and $g\bigl(\Psi^{(t)}\bigr)$ is
bounded from above.  As in the proof of (i), we can take a subsequence
of the form (\ref{eq:Psi}) with $\Psi_1=\Psi^0$ and
$\epsilon_1^{(t)}\to 1$.  Under the assumption (\ref{eq:cond2}),
the lower bound (\ref{eq:below}) of $e^{g(\Psi^{(t)})}$ always
diverges to infinity, which is a contradiction. 
\end{proof}

\subsection{Minimal ranks}

Assume, as throughout, that we have data matrices
$Y_1,\dots,Y_n\in\mathbb{R}^{m_1\times m_2}$ with $m_1\ge m_2$.  For
$k=1,\dots, m_2$, define the minimal rank
\begin{equation}
  \label{eq:minimal-rank}
  r_n(m_1,m_2,k) \;=\; \min\left\{\rk\big(Y_1X,\dots,Y_nX \big) :
    X\in\mathbb{R}^{m_2\times k},\; \rk(X)=k \right\}.
\end{equation}
  Now define
\begin{equation}
  \label{eq:minimal-rank-summary}
  S_n(m_1,m_2) = \min_{1\le k<m_2}\bigl\{ m_2 r_n(m_1,m_2,k) - m_1 k \bigr\}. 
\end{equation}
For generic data matrices, the results in this section can be summarized as follows.

\begin{theorem}
  \label{thm:Snm1m2}
  The likelihood function is a.s.~bounded if and only if
  a.s.~$S_n(m_1,m_2)\ge 0$.  The Kronecker MLE
  exists uniquely a.s.~if and only if a.s.~$S_n(m_1,m_2)> 0$.
\end{theorem}

\section{Square matrices}
\label{sec:square-matrices}

This section treats the case of square data matrices.  So, $m_1=m_2$
and we denote this common value also by $m$.   The
results we develop, specifically,
Corollaries~\ref{cor:n1-square},~\ref{cor:square-Nu}, and~\ref{cor:Nb-square}, yield the
following statement about the three sample size thresholds.

\begin{theorem}
  \label{thm:square-thresholds}
  For square data matrices of any size $m\ge 2$, it holds that
  $N_e(m,m)=N_b(m,m)=1$ and $N_u(m,m)=3$.
\end{theorem}

Throughout the remainder of this section we tacitly assume that each
one of the data matrices $Y_1,\dots,Y_n$ is invertible, as is the case almost
surely.

\subsection{One square data matrix}
\label{sec:square-rank-drops}

We begin with an observation utilized in
\cite{volfovsky:2015}.

\begin{proposition}
  \label{prop:square-n1}
  If $m_1=m_2$ and $n=1$, then the profile likelihood function $g$
  from~(\ref{eq:profile-neglik}) is constant.
\end{proposition}
\begin{proof}
  The single $m\times m$ data matrix $Y_1$ being invertible, we have
  \[
    g(\Psi)\;=\; m\log\det(Y_1\Psi Y_1^T)-m\log\det(\Psi) \;=\;
    2m\log|\det(Y_1)|,
  \]
  which does not depend on $\Psi$.
\end{proof}

The proposition implies that for $n=1$ the likelihood function
achieves its maximum but not uniquely so.  We may deduce from the rank
conditions in Section~\ref{sec:algebr-cond-exist} that
$r_1(m,m,k)\ge k$ for all $m\ge 2$ and $k=1,\dots,m$.  As
$r_n(m_1,m_2,k)$ is non-decreasing in $n$, we obtain that
$r_n(m,m,k)\ge k$ always, which implies $S_n(m,m)\ge 0$.  By
Theorem~\ref{thm:Snm1m2}(i), we have:
\begin{corollary}
  \label{cor:n1-square}
  The  boundedness threshold of square matrices of size $m\ge 2$ is $N_b(m,m)=1$.
\end{corollary}

\subsection{Two square data matrices}

Moving to the case of $n=2$ square data matrices, we first provide
detail on the ranks $r_2(m,m,k)$.

\begin{lemma}
\label{lm:r2mmk}
Let $Y_1,Y_2\in\mathbb{R}^{m\times m}$ be generic, and let
$1\le k\le m$.  If $Y_1^{-1}Y_2$ has a real eigenvalue, or if $k$ is
even, then
\[
 r_2(m,m,k) = k.
\]
If all eigenvalues of $Y_1^{-1}Y_2$ are complex and if $k$ is odd,
then
\[
 r_2(m,m,k) = k+1.
\]
\end{lemma}

\begin{proof} 
Let $W=Y_1^{-1}Y_2$.  
We evaluate
\[
 r_2(m,m,k) = \min_{\rk(X)=k}\rk\begin{pmatrix} Y_1X,  Y_2X \end{pmatrix}
= \min_{\rk(X)=k}\rk\begin{pmatrix} X, WX \end{pmatrix}.
\]
Since $W$ is real and generic, its characteristic function does not
have multiple zeros and all of its (complex) Jordan blocks are of size
1.  For any real eigenvalue $\lambda_j$, let
$x_j\in\mathbb{R}^{m}\setminus\{0\}$ be an associated real eigenvector
such that
\[
 W x_j = \lambda_j x_j.
\]
For any pair of complex eigenvalues $\mu_j \pm i \nu_j$, let
$y_j,z_j\in\mathbb{R}^m\setminus\{0\}$ be real vectors such that
$y_j \pm i z_j\in \mathbb{C}^{m}$ are eigenvectors corresponding to
$\mu_j \pm i \nu_j$, so
\[
 W(y_j \pm i z_j) = (\mu_j \pm i \nu_j)(y_j \pm i z_j)
\iff
  W\begin{pmatrix} y_j , z_j \end{pmatrix} = \begin{pmatrix} y_j , z_j \end{pmatrix}\begin{pmatrix} \mu_j & \nu_j \\ -\nu_j & \mu_j \end{pmatrix}.
\]
Altogether the vectors $x_j$, $y_j$, and $z_j$ form a basis of $\mathbb{R}^m$.

Suppose now that $W$ has at least one real eigenvalue or $k$ is even.
Choose a full rank matrix $X^*\in\mathbb{R}^{m\times k}$ by selecting
its columns as individual vectors $x_j$ or pairs $(y_j,z_j)$, that is,
\[
 X^* = (x_1,\ldots,x_s,y_1,z_1,\ldots,y_t,z_t),\quad s+2t=k.
\]
Then
\[
  r_2(m,m,k) \le \rk\begin{pmatrix} X^*, WX^* \end{pmatrix} = k.
\]
On the other hand, trivially we have
\[
r_2(m,m,k)= \min_{\rk(X)=k}\rk\begin{pmatrix} X, WX \end{pmatrix} \ge \min_{\rk(X)=k}\rk(X)=k.
\]

In the remaining case where $k$ is odd and all eigenvalues complex, we
may reduce the rank of $(X,WX)$ to $k+1$ by choosing choose $k-1$
columns of $X$ based on eigenvectors.  However, we cannot reduce rank
further as this would contradict the linear independence of
eigenvectors.
\end{proof}

\begin{lemma}
Square matrices of size $m\ge 2$ have
\[
  S_2(m,m) = \begin{cases}
    2 & \mbox{if $m=2$ and $Y_1^{-1}Y_2$ has complex eigenvalues}, \\
    0 & \mbox{otherwise}.
    \end{cases}
\]
\end{lemma}

\begin{proof}
When $m=2$ and $Y_1^{-1}Y_2$ has complex eigenvalues, $r_2(2,2,1)=1+1$, and
$S_2(2,2) = \min_{1\le k\le 2-1} \{ 2 r_2(2,2,k)-2 k \} = 2 r_2(2,2,1)-2\times 1=2$.

When $m=2$ and $Y_1^{-1}Y_2$ has real eigenvalues, $r_2(2,2,1)=1$, and
$S_2(2,2) = 2 r_2(2,2,1)-2\times 1=0$.

When $m\ge 3$, because of $r_2(m,m,k)\ge k$, $S_2(m,m) = \min_{1\le k\le m-1} \left\{ m r_2(m,m,k)-m k \right\} \ge 0$ holds, and the equality always attains at $k=2$, $r_2(m,m,k)=r_2(m,m,2)=2$.
\end{proof}

\begin{corollary}
  \label{cor:square-Nu}
  The  uniqueness threshold of  square matrices of size $m\ge 2$ is
  $N_u(m,m)=3$.
\end{corollary}
\begin{proof}
  We know from Proposition~\ref{prop:known-results} that
  $N_u(m,m)\le 3$.  If $m\ge 3$, then Lemma~\ref{lm:r2mmk} yields that
  $S_2(m,m)=0$ generically and thus a sample size of $n=2$ is not
  sufficient for a.s.~unique existence of the Kronecker MLE.  Hence,
  $N_u(m,m)= 3$.

  If $m=2$, then there is the subtlety that $S_2(m,m)=0$ if
  $Y_1^{-1}Y_2$ has real eigenvalues, and $S_2(m,m)=1$ if
  $Y_1^{-1}Y_2$ has complex eigenvalues.  However, as either case
  arises with positive probability, $N_u(m,m)=3$ also for $m=2$.
\end{proof}

\begin{remark}
  Contrasting cases (i) and (iii) in
  Proposition~\ref{prop:square-2x2}, we see that the minimal rank
  $r_n(m_1,m_2,k)$ may change when we minimize over complex instead of
  real matrices.
\end{remark}

\subsection{Achieving maximum likelihood for two square data matrices}

As we know that $N_u(m,m)=3$ and $N_b(m,m)=1$, with maximum achieved, there only remains the question  whether the (bounded) likelihood function achieves
its maximum for a sample of $n=2$ invertible data matrices.  We begin
with the smallest case of $m=2$, which exhibits exceptional behavior
as noted in the proof of Corollary~\ref{cor:square-Nu}; see also
\cite[Section 4.3]{soloveychik:2016}.  The following proposition
considers all possible cases and gives their probabilities.

\begin{proposition}
  \label{prop:square-2x2}
  Suppose $n=2$ with invertible data matrices $Y_1$ and $Y_2$ of
  size $2\times 2$.   Three cases are possible:
  \begin{enumerate}
  \item[(i)] The matrix $Y_1^{-1}Y_2$ has real eigenvalues and is
    diagonalizable.  The likelihood function is then bounded and
    achieves its maximum, but not uniquely so.
  \item[(ii)] The matrix $Y_1^{-1}Y_2$ has real eigenvalues but is not
    diagonalizable.  The likelihood function is then bounded but fails
    to achieve its maximum.
  \item[(iii)] The eigenvalues of $W=(w_{jk})=Y_1^{-1}Y_2$ are
    complex.  The Kronecker MLE then exists uniquely and is given by
    any positive definite matrix of the form 
    \[
      \Psi \;=\; 
      \lambda\begin{pmatrix}
        w_{12} & \frac{1}{2}(w_{22}-w_{11})\\
        \frac{1}{2}(w_{22}-w_{11}) & -w_{21}
    \end{pmatrix}, \qquad \lambda\in\mathbb{R}.
    \]
  \end{enumerate}
  Only cases (i) and (iii) occur with positive probability.  If the
  entries of $Y_1$ and $Y_2$ are i.i.d.~$\mathcal N(0,1)$, then case (i) has
  probability $\pi/4\approx 0.7854$.
\end{proposition}

\begin{proof}
  We may put $Y_1$ and $Y_2$ in special position through the action of
  $\mathrm{GL}(m,\mathbb{R})\times\mathrm{GL}(m,\mathbb{R})$ discussed
  in Section~\ref{sec:group-action}.  With $A=B^{-1}Y_1^{-1}$, we have
\begin{align*}
  AY_1B &= I_m, \qquad AY_2B = B^{-1}Y_1^{-1}Y_2B.
\end{align*}
Now choose $B$ such that $AY_2B$ becomes the real-valued
Jordan form of $W=Y_1^{-1}Y_2$.
  
Case (iii): If the two eigenvalues are complex then $S_2(2,2)=1$ as
noted in the proof of Corollary~\ref{cor:square-Nu}.  By
Theorem~\ref{thm:Snm1m2} the Kronecker MLE exists uniquely.  In
special form, our data matrices take the form
\begin{align}
  \label{eq:complex-Jordan}
  Y_1= I_2, \qquad  Y_2=
    \begin{pmatrix}
      a  & b \\
      -b & a
    \end{pmatrix}
\end{align}  
with $a,b\in\mathbb{R}$.  Then the negated profile log-likelihood
function from~(\ref{eq:profile-neglik}) is
  \[
    g(\Psi)\;=\; 2\log\left( 1+2a^2+a^4+b^4 + 2 a^2 b^2\; +\; b^2\frac{\|\Psi\|_F^2}{\det(\Psi)}\right).
  \]
  Let $\lambda_1\ge\lambda_2>0$ be the two eigenvalues of the positive
  definite $2\times 2$ matrix $\Psi$.   Then
  \[
    \frac{\|\Psi\|_F^2}{\det(\Psi)} \;=\; \left( \frac{\lambda_1}{\lambda_2}\right)^2\;+\;\left(\frac{\lambda_2}{\lambda_1} \right)^2
  \]
  is minimal when $\lambda_1=\lambda_2$, which occurs if and only if
  $\Psi=\lambda I_2$ for $\lambda>0$.  Translating back to the
  original data yields the claimed formula for the MLE.

  Case (i):  By the diagonalizability assumption, the special form of
  our data matrices is 
 \begin{align*}
  Y_1= I_2, \qquad Y_2=
    \begin{pmatrix}
      a & 0 \\
      0 & b
    \end{pmatrix}
 \end{align*}
  with $a,b\in\mathbb{R}$.  The negated profile log-likelihood
  function from~(\ref{eq:profile-neglik}) now equals
  \begin{equation}
    \label{eq:g:two-real-evs-diag}
   g(\Psi)\;=\;2\log\left((1 + a b)^2  + (a-b)^2\frac{\psi_{11}
     \psi_{22}}{\det(\Psi)}\right).
  \end{equation}
  Let $\rho\equiv\rho(\Psi)=\psi_{12}/\sqrt{\psi_{11}\psi_{22}}$ be
  the correlation.  Then
  \[
    \frac{\psi_{11}
      \psi_{22}}{\det(\Psi)} \;=\; \frac{1}{1-\rho^2} 
  \]
  is minimized uniquely for $\rho=0$.  Hence, the function $g$
  from~(\ref{eq:g:two-real-evs-diag}) is minimized by all diagonal
  matrices.  The likelihood function achieves its maximum but
  not uniquely so.

  Case (ii): The special form of our data matrices is now
  \begin{align*}
    Y_1= I_2, \qquad Y_2=
                 \begin{pmatrix}
                   a & 1 \\
                   0 & a
                 \end{pmatrix}
  \end{align*}
  with $a\in\mathbb{R}$.  The negated profile log-likelihood function
  from~(\ref{eq:profile-neglik}) is
  \begin{equation}
    \label{eq:g:two-real-evs-nondiag}
    g(\Psi)\;=\;2\log\left((1 + a^2)^2  +
      \frac{\psi_{22}^2}{\det(\Psi)}\right)\;>\; 2\log\left((1 + a^2)^2\right)
  \end{equation}
  as $\psi_{22},\det(\Psi)>0$.  If we fix the values $\psi_{11}=1$ and
  $\psi_{12}=0$, and let $\psi_{22}\to 0$, then $g(\Psi)$ converges to
  $2\log((1 + a^2)^2)$.  Hence,
  \[
    2\log((1 + a^2)^2) \;=\; \inf\{ g(\Psi):\Psi\in\mathit{PD}(2)\}
  \]
  but this infimum is not achieved.

  Finally, case (ii) occurs with probability zero as $Y_1^{-1}Y_2$ has
  to have an eigenvalue of
  multiplicity two.  The
  probability of case (i) is found in Lemma~\ref{lem:ev-prob} in the
  appendix.
\end{proof}

The dichotomy from the case of $2\times 2$ matrices disappears for
larger matrices.

\begin{proposition}
  If $n=2$ with square data matrices of size $m\ge 3$, then the
  likelihood function is a.s.~bounded and achieves its maximum, but
  not uniquely so.
\end{proposition}
\begin{proof}
  We prove the claim by exhibiting an at least two-dimensional set of
  critical points, which must all be global optima by
  Lemma~\ref{lem:gconvex}.

  As in the proof of Proposition~\ref{prop:square-2x2}, assume that
  $Y_1=I_m$ is the identity and that $Y_2$ is in real Jordan form for
  the almost surely occurring case of all eigenvalues being distinct.
  Then $Y_2$ is block-diagonal with blocks of size 1 or 2; the
  $2\times 2$ blocks are as in~(\ref{eq:complex-Jordan}).  As the
  matrix size is $m\ge 3$, there are $k\ge 2$ blocks, which we denote
  by $Y_{21},\dots,Y_{2k}$.  Let $b_1,\dots,b_k\in\{1,2\}$ be the
  sizes of these blocks.

  The profile function we minimize is
  \begin{equation}
    \label{eq:square-g-Jordan}
    g(\Psi)\;=\;m\log\det\left(\Psi+Y_2\Psi Y_2^T
    \right)-m\log\det(\Psi), \qquad \Psi\in\mathit{PD}(m).
  \end{equation}
  For each block define an analogous function
  \begin{equation*}
    g_l(\Psi_{l})=b_l\log\det\left(\Psi_l+Y_{2l}\Psi_l Y_{2l}^T
    \right)-b_l\log\det(\Psi_l), \qquad \Psi_l\in\mathit{PD}(b_l).
  \end{equation*}
  The logarithm of the determinant has differential
  \[
    \mathrm{d}\log\det(\Psi) = \tr\left( \Psi^{-1} \mathrm{d}\Psi\right).
  \]
  It follows that the differential of $g$
  in~(\ref{eq:square-g-Jordan}) is
  \begin{equation*}
    \mathrm{d}g(\Psi;U) 
    = \\
    m\tr\left\{ \left( \Psi+Y_2\Psi Y_2^T\right)^{-1}U
    \right\} - m \tr\left( \Psi^{-1}U\right).
  \end{equation*}
  As candidates, consider block-diagonal matrices
  $\Psi_0$ with $k$ blocks $\Psi_{01},\dots,\Psi_{0k}$ of sizes
  $b_1,\dots,b_k$, respectively.  Then $\Psi_0+Y_2\Psi_0 Y_2^T$ is
  block-diagonal, and we have
  \begin{equation*}
    \frac{1}{m}\mathrm{d}g(\Psi_0;U) = \sum_{l=1}^k
    \frac{1}{b_l}\mathrm{d}g_l(\Psi_{0l};U_l).
  \end{equation*}
  Now, take each block of $\Psi_0$ to be a multiple of the identity,
  so $\Psi_{0l}=\lambda_l I_{b_l}$ for $l=1,\dots,k$.  If $b_l=1$,
  then $\mathrm{d}g_l(\Psi_{0l};U_l)=0$ trivially because $g_l$ is
  then constant. If $b_l=2$, then $\mathrm{d}g_l(\Psi_{0l};U_l)=0$ as
  we showed in the proof of Proposition~\ref{prop:square-2x2} that
  $g_l$ is minimized by multiples of $I_2$.  We conclude that
  block-diagonal matrices with blocks equal to multiples of the
  identity are critical points.  As there are $k\ge 2$ blocks the
  critical points we exhibited form a set of dimension at least
  2. Hence, the likelihood function achieves its maximum, but not
  uniquely so.
\end{proof}

\begin{corollary}
  \label{cor:Nb-square}
  The existence threshold of square matrices of size $m\ge 2$ is\\ $N_e(m,m)=1$.
\end{corollary}

\section{Rectangular matrices}
\label{sec:rectangular-matrices}

In this section we consider $n=2$ rectangular matrices $Y_1$ and $Y_2$ of size $m_1\times m_2$ with $m_1>m_2$.  As discussed in
Section~\ref{sec:introduction}, the nontrivial case is then
$n m_2=2 m_2>m_1> m_2$.
For this case, we derive explicit solutions for the minimal rank
$r_2(m_1,m_2,k)$ in (\ref{eq:minimal-rank}) and $S_2(m_1,m_2)$ in
(\ref{eq:minimal-rank-summary}).

\subsection{Kronecker canonical form}

As discussed in Section~\ref{sec:group-action}, our problem is
invariant with respect to the group action $Y_i \mapsto A Y_i B$,
$(A,B)\in\mathrm{GL}(m_1,\mathbb{R})\times\mathrm{GL}(m_2,\mathbb{R})$.  The theorem below
states that when $Y_1$ and $Y_2$ are generic, by choosing $A$ and $B$
appropriately (depending on the data $Y_i$), the problem can be reduced
into a simplified canonical form.

In the sequel, we write $0_{k,l}$ for the $k\times l$ matrix with all entries zero.  

\begin{theorem}
\label{thm:tenberge}
Let $Y_1$ and $Y_2$ be generic rectangular matrices of size $m_1\times
m_2$ with $2 m_2>m_1>m_2$.   
There exist $A\in\mathrm{GL}(m_1,\mathbb{R})$ and $B\in\mathrm{GL}(m_2,\mathbb{R})$ such that
\begin{equation}
\label{eq:AYB}
 A Y_1 B = \begin{pmatrix} I_{m_2} \\ 0_{m_1-m_2, m_2} \end{pmatrix}, \qquad
 A Y_2 B = \begin{pmatrix} 0_{m_1-m_2, m_2} \\ I_{m_2} \end{pmatrix}.
\end{equation}
\end{theorem}

\begin{proof}
This is a variation of the Kronecker canonical form (see Remark~\ref{remark:kroneckerform}). 
For constructive proofs, see \cite{tenberge-kiers:1999} and \cite[Theorem 5.1.8]{murota}.
\end{proof}

\begin{remark}
\label{remark:kroneckerform}
Let
\[
 U_l=\begin{pmatrix}
 I_l \\
 0_{1,l}
 \end{pmatrix}\in\mathbb{R}^{(l+1)\times l}, \qquad
 L_l=\begin{pmatrix}
 0_{1,l} \\
 I_l
 \end{pmatrix}\in\mathbb{R}^{(l+1)\times l}.
\]
It is known that for generic matrices $Y_1$ and $Y_2$ in Theorem~\ref{thm:tenberge}, there exist $A\in\mathrm{GL}(m_1,\mathbb{R})$ and $B\in\mathrm{GL}(m_2,\mathbb{R})$ such that
\begin{equation}
\label{eq:AYB'}
 A Y_1 B = \mathrm{diag}(\underbrace{U_{l+1},\ldots,U_{l+1}}_{n_a},\underbrace{U_{l},\ldots,U_{l}}_{n_b}) , \qquad
 A Y_2 B =  \mathrm{diag}(\underbrace{L_{l+1},\ldots,L_{l+1}}_{n_a},\underbrace{L_{l},\ldots,L_{l}}_{n_b}),
\end{equation}
where $l$, $n_a$ and $n_b$ are functions of $(m_1,m_2)$ defined in (\ref{eq:l}) and (\ref{eq:nanb}); see \cite[Section 3.3]{edelman:1997}.
The pair of block diagonal matrices in (\ref{eq:AYB'}) is referred to as the Kronecker canonical form. 
We easily see that the form in (\ref{eq:AYB}) may be obtained from that in (\ref{eq:AYB'}) by permuting rows and columns.
\end{remark}

  In the remainder of this section, we assume without loss of
  generality that $Y_1$ and $Y_2$ are already in the canonical form in
  (\ref{eq:AYB}).  The minimal rank (\ref{eq:minimal-rank}) we will
  determine then becomes 
\begin{align*}
 r_2(m_1,m_2,k)
 \;=&\; \min\left\{\rk\begin{pmatrix} X Y_1^T \\ X Y_2^T \end{pmatrix} : X\in\mathbb{R}^{k\times m_2},\; \rk(X)=k \right\} \\
 \;=&\; \min\left\{\rk \XOOX_{2 k\times m_1} : X\in\mathbb{R}^{k\times m_2},\; \rk(X)=k \right\},
\end{align*}
where for ease of presentation the matrix whose rank we consider has
been transposed.

\subsection{Gr\"obner basis computation}  

When $m_2$ is small, the minimal rank $r_2(m_1,m_2,k)$ can be found by algebraic computations.
Since
\begin{align*}
 \rk\XOOX
 =
\rk\left(\!\begin{array}{c}\begin{array}{c|c}S X\ \ & 0 \end{array} \\ \hline \begin{array}{c|c}0 &\ \ S X\end{array}\end{array}\!\right), \quad S\in\mathrm{GL}(k,\mathbb{R}),
\end{align*}
we can set $k$ columns of $X\in\mathbb{R}^{k\times m_2}$ to form an identity matrix.  We thus proceed through the following steps:
\begin{itemize}
\item[Step 1.]
For each set $\{l_1,\ldots,l_k\}$ with $1\le l_1<\cdots< l_k\le m_2$, repeat (i) and (ii) below:

\begin{itemize}
\item[(i)]
Let
\[
 X =
 \begin{pmatrix}
 x_{11} & \cdots & x_{1 m_2} \\
 \vdots &        & \vdots \\
 x_{k1} & \cdots & x_{k m_2}
 \end{pmatrix}
 \ \mbox{with}\ %
 \begin{pmatrix}
 x_{1 l_1} & \cdots & x_{1 l_k} \\
 \vdots &        & \vdots \\
 x_{k l_1} & \cdots & x_{k l_k}
 \end{pmatrix}
=
 \begin{pmatrix}
 1 &        & 0 \\
   & \ddots \\
 0 &        & 1
 \end{pmatrix}.
\]

\item[(ii)]
For $i=0,1,\ldots$, try to solve the polynomial system
\begin{equation}
\label{eq:minors}
 \mbox{all $(2k-i)\times (2k-i)$ minors of } 
 \left(\begin{array}{ccc} \multicolumn{2}{c}{X} & 0_{k,m_1-m_2} \\ 0_{k,m_1-m_2} & \multicolumn{2}{c}{X} \end{array}\right)
= 0
\end{equation}
by computing and inspecting a Gr\"{o}bner basis.
If a real solution $X^*$ exists for $i=i^*$ but not for $i=i^*+1$, let
\begin{equation}
\label{eq:min-rank}
  \mathtt{Rank}[\{l_1,\ldots,l_k\}]:=\rk\left(\!\begin{array}{c}\begin{array}{c|c} X^*\ \ & 0 \end{array} \\ \hline \begin{array}{c|c}0 &\ \ X^*\end{array}\end{array}\!\right).
\end{equation}
\end{itemize}

\item[Step 2.]
Take the minimum for all possible $1\le l_1<\cdots<l_k\le m_2$:
\[
 r_2(m_1,m_2,k) = \min_{1\le l_1<\cdots<l_k\le m_2} \mathtt{Rank}[\{l_1,\ldots,l_k\}].
\]
\end{itemize}

\begin{example}
Suppose $m_1=5$, $m_2=3$, $k=2$.
For $\{l_1,l_2\}=\{1,3\}$,
$X=\left(\begin{smallmatrix} 1 & x_{12} & 0 \\ 0 & x_{22} & 1 \end{smallmatrix}\right)$,
and
\begin{equation}
  \label{eq:GB-example}
  \XOOX_{4\times 5} =
\begin{pmatrix} 1 & x_{12} & 0 & 0 & 0 \\ 0 & x_{22} & 1 & 0 & 0 \\ 0 & 0 & 1 & x_{12} & 0 \\ 0 & 0 & 0 & x_{22} & 1 \end{pmatrix}.
\end{equation}  
The first, third, and fifth column of this $4\times 5$ matrix are linearly independent so its rank cannot drop below 3.  This is reflected in the $3\times 3$ minors being $\{x_{12},x_{22},x_{12}^2,x_{22}^2,x_{12} x_{22},1,0\}$, with Gr\"{o}bner basis $\{1\}$ and no solution (real or complex) for  (\ref{eq:minors}) when $i=1$.  However, for $i=0$, the set of $4\times 4$ minors of the matrix
is
$\{x_{12},x_{22},x_{12}^2,x_{22}^2,x_{12} x_{22} \}$ with a Gr\"{o}bner basis being $\{x_{12},x_{22}\}$.  This confirms that $x_{12}=x_{22}=0$ is the (evident) solution for (\ref{eq:minors}) when $i=0$.  Our procedure concludes
\begin{equation*}
 \mathtt{Rank}[\{1,3\}]=\rk \XOOX \bigg|_{x_{12}=x_{22}=0}
= \rk\left(\!\begin{array}{c} \begin{array}{ccc|cc} 1 & 0 & 0 & 0 & 0 \\ 0 & 0 & 1 & 0 & 0 \end{array} \\ \hline \begin{array}{cc|ccc} 0 & 0 & 1 & 0 & 0 \\ 0 & 0 & 0 & 0 & 1 \end{array}\end{array}\!\right)
 =3 < 2k=4.
\end{equation*}
We observe a drop in rank.  For the other combinations $\{l_1,l_2\}=\{1,2\}$ and $\{2,3\}$, no rank drop occurs, i.e., $\mathtt{Rank}[\{l_1,l_2\}]=4$.
Hence, $r_2(m_1,m_2,k)=r_2(5,3,2)=3$.
\end{example}

In the example just given a well-devised 0-1 matrix $X$ attains the minimal rank.
We shall see that such a matrix exists for general $(m_1,m_2,k)$; see the construction in (\ref{eq:XX}).

\subsection{Evaluation of the minimal rank} 

Our strategy to determine $r_2(m_1,m_2,k)$ is to first provide an upper bound by specifying a special 0-1 matrix $X$.  We then prove that no other matrix can achieve lower rank than $X$.  In order to state our results, some further notation is needed.

Let
\begin{align}
\label{eq:l}
 l(m_1,m_2) =& \max\{ l\in\mathbb{N} \mid (l+1) m_2 - l m_1 > 0 \} 
= \Bigl\lceil \frac{m_2}{m_1-m_2}\Bigr\rceil-1 \ge 1.
\end{align}
Based on the value $l(m_1,m_2)$,
the set of pairs $(m_1,m_2)$ of interest is disjointly divided as
\begin{multline*}
 \bigl\{ (m_1,m_2) \mid m_1 < 2 m_2 \bigr\}
= \bigsqcup_{l\ge 1} \bigl\{ (m_1,m_2) \mid l=l(m_1,m_2) \bigr\} \\
= \bigsqcup_{l\ge 1} \bigl\{ (m_1,m_2) \mid (l+1) m_2 - l m_1 > 0, \ (l+1) m_1 - (l+2) m_2 \ge 0 \bigr\}.
\end{multline*}
Let
\begin{equation}
\label{eq:nanb}
 n_a = (l+1)m_2-l m_1 >0, \quad n_b = (l+1)m_1-(l+2)m_2 \ge 0 \quad \mbox{with }l=l(m_1,m_2).
\end{equation}
Now we partition the columns of $X$ as
\begin{equation*}
 X = \Bigl(\underbrace{X_1}_{n_a},\underbrace{X_2}_{n_b},\underbrace{X_3}_{n_a},\underbrace{X_4}_{n_b},\ldots,\underbrace{X_{2l+1}}_{n_a}\Bigr)_{k\times m_2}.
\end{equation*}
Indeed, the number of columns of $X$ is
\[
 (l+1) n_a + l n_b = m_2.
\]
Accordingly,
\begin{equation}
\label{eq:I2XY}
\XOOX_{2k\times m_1}
 = \begin{pmatrix} X_1 & X_2 & X_3 & X_4 & \cdots & X_{2l+1} & 0 & 0 \\
 0 & 0 & X_1 & X_2 & \cdots & X_{2l-1} & X_{2l} & X_{2l+1} \end{pmatrix}_{2k\times m_1}.
\end{equation}

\begin{theorem}
\label{thm:r2}
For given $(m_1,m_2)$ with $2m_2> m_1>m_2$, let $l=l(m_1,m_2)$, $n_a$, and $n_b$ be defined as in (\ref{eq:l}) and (\ref{eq:nanb}).  Then the minimal rank
$r_2(m_1,m_2,k)$ is the solution of the integer programming problem
\begin{equation}
\label{eq:r2}
 r_2(m_1,m_2,k) = 
 \min\bigl\{a_1+b_1+k \mid k\le a_1 (l+1) + b_1 l,\ 0\le a_1\le n_a,\ 0\le b_1\le n_b\bigr\},
\end{equation}
where $a_1$ and $b_1$ are non-negative integers.
\end{theorem}

\begin{proof}
{[Upper bound]}
We first show that the right-hand side of (\ref{eq:r2}) is an upper bound on $r_2(m_1,m_2,k)$.  To do this, we specify a particular matrix $X$ that gives a rank equal to the right-hand side of (\ref{eq:r2}).

Let $n_a\ge a_1\ge\cdots\ge a_{l+1}\ge 0$ and $n_b\ge b_1\ge\cdots\ge b_{l}\ge 0$ such that $k_a=\sum_{j=1}^{l+1} a_j$, $k_b=\sum_{j=1}^{l} b_j$ with $k_a+k_b=k$.
For integer $i\ge 1$, define the matrices 
\begin{equation}
\label{eq:XX}
 X_{2i-1} =
 \left( \begin{array}{l} 0_{\sum_{j=1}^{i-1} a_{j} , n_a} \\[3mm] \bigl(I_{a_i},0_{a_i,n_a-a_i}\bigr) \\[2mm] 0_{\sum_{j=i+1}^{l+1} a_{j}, n_a} \\[3mm] 0_{k_b, n_a} \end{array} \right),  
 \qquad
 X_{2i} =
 \left( \begin{array}{l} 0_{k_a, n_b} \\[2mm] 0_{\sum_{j=1}^{i-1} b_{j}, n_b} \\[3mm] \bigl(I_{b_i},0_{b_i, n_b-b_i}\bigr) \\[2mm] 0_{\sum_{j=i+1}^{l} b_{j}, n_b} \end{array}\right). 
\end{equation}
The matrices $X_{2i-1}$ are of size $k\times n_a$ and defined for $i\le l+1$.  The matrices $X_{2i}$ are of size $k\times n_b$ and defined for $i\le l$.
Then,
\begin{equation*}
 \rk\begin{pmatrix} X_1 & X_2 & X_3 & X_4 & \cdots & X_{2l+1} & 0 & 0 \\
 0 & 0 & X_1 & X_2 & \cdots & X_{2l-1} & X_{2l} & X_{2l+1} \end{pmatrix}_{2k\times m_1} =
 a_1+b_1+k.
\end{equation*}
We may now minimize this rank $a_1+b_1+k$ by varying $a_1,\dots,a_{l+1}$ and $b_1,\dots,b_l$.  Our claim is then that this optimization over $2l+1$ variables gives the minimum on the right-hand side of (\ref{eq:r2}).  To show this, we first show that
\begin{multline}
\label{eq:equivalence}
 \Bigl\{ (a_1,b_1) \mid n_a\ge a_{1}\ge\cdots\ge a_{l+1}\ge 0,\ n_b\ge b_{1}\ge\cdots\ge b_{l}\ge 0,\ k = \sum_i a_i + \sum_j b_j \Bigr\} \\
=\Bigl\{  (a_1,b_1) \mid 0\le a_1\le n_a,\ 0\le b_1\le n_b,\ a_1+b_1\le k \le a_1(l+1) + b_1 l \Bigr\}. 
\end{multline}

The inclusion ``$\subset$'' for the two sets in \eqref{eq:equivalence} is obvious.  To prove ``$\supset$'', let $(a_1,b_1)$ be a point in the set on the right-hand side of (\ref{eq:equivalence}).  We then need to argue that there exist $a_2,\ldots,a_{l+1},b_2,\ldots,b_{l}$ such that
$a_1\ge a_2\ge\cdots\ge a_{l+1}\ge 0$, $b_1\ge b_2\ge\cdots\ge b_{l}\ge 0$, and $\sum_{i=2}^{l+1} a_i + \sum_{i=2}^{l} b_i = k-a_1-b_1$ hold.
But this is obvious because 
\begin{align*}
& \biggl\{\sum_{i=2}^{l+1} a_i + \sum_{j=2}^{l} b_j \mid a_1\ge a_2\ge\cdots\ge a_{l+1}\ge 0,\ b_1\ge b_2\ge\cdots\ge b_{l}\ge 0 \biggr\} \\
&= \biggl\{\sum_{i=2}^{l+1} a_i + \sum_{j=2}^{l} b_j \mid 0\le a_2,\ldots,a_{l+1}\le a_1,\ 0\le b_2,\ldots,b_{l}\le b_1 \biggr\}\\
&= \bigl\{0,1,\ldots,a_1 l+b_1(l-1) \bigr\}
\end{align*}
covers all possible values for $k-a_1-b_1$.  

Finally, the feasible set for the minimization in (\ref{eq:r2}) differs from the set on the right-hand side of (\ref{eq:equivalence}) only by the constraint $a_1+b_1\le k$.  However, a pair $(a_1,b_1)$ cannot be a minimizer for (\ref{eq:r2}) if $a_1+b_1>k$.  Hence, the minimum in (\ref{eq:r2}) equals the minimum over the right-hand side of (\ref{eq:equivalence}).

\bigskip\noindent
{[Lower bound]}
To obtain a lower bound, it is convenient to rearrange the columns of (\ref{eq:I2XY}) as
\begin{equation}
\label{eq:XO}
 \begin{pmatrix}
 X_1 & X_3 & \cdots & X_{2l+1} & 0 & X_2 & X_4 & \cdots & X_{2l} & 0 \\
   0 & X_1 & \cdots & X_{2l-1} & X_{2l+1} & 0 & X_2 & \cdots & X_{2l-2} & X_{2l} \end{pmatrix}_{2k\times m_1}.
\end{equation}
We first find a lower bound of the rank of (\ref{eq:XO}) for a fixed matrix $X$, and then obtain a lower bound for all possible $X$.

Let $k_a = \rk(X_1,X_3,\ldots,X_{2l+1})$, and $k_b=k-k_a$.
There is a $(k-k_a)\times k$ matrix $T$ such that $T X=T(X_1,X_3,\ldots,X_{2l+1})=0$.
Let $S$ be a $k_a\times k$ matrix such that
$\begin{pmatrix} S \\ T \end{pmatrix}\in\mathrm{GL}(k,\mathbb{R})$.
Multiplying the matrix
\[
 \begin{pmatrix} S & 0 \\0 & S \\ T & 0 \\ 0 & T \end{pmatrix} \in\mathrm{GL}(2k,\mathbb{R})
\]
 to (\ref{eq:XO}) from the left yields
\[
 \begin{pmatrix}
 \X{1}_1 & \X{1}_3 & \cdots & \X{1}_{2l+1} & 0        & * & * & \cdots & * & 0\\
   0 & \X{1}_1 & \cdots & \X{1}_{2l-1} & \X{1}_{2l+1} & 0 & * & \cdots & * & * \\
   0 & 0 & \cdots & 0 & 0 & \X{1}_2 & \X{1}_4 & \cdots & \X{1}_{2l} & 0 \\
   0 & 0 & \cdots & 0 & 0 & 0   & \X{1}_2 & \cdots & \X{1}_{2l-2} & \X{1}_{2l}
\end{pmatrix}_{2k\times m_1},
\]
where we define $\X{1}_{2i-1}=S X_{2i-1}$ and $\X{1}_{2i}=T X_{2i}$.
The rank of (\ref{eq:XO}) is bounded below by the sum of the ranks of the two matrices:
\begin{equation}
\label{eq:Xodd}
 \begin{pmatrix}
 \X{1}_1 & \X{1}_3 & \cdots & \X{1}_{2l+1} & 0        \\
   0 & \X{1}_1 & \cdots & \X{1}_{2l-1} & \X{1}_{2l+1}
\end{pmatrix}_{2k_a\times n_a(l+2)}
\end{equation}
and
\begin{equation}
\label{eq:Xeven}
 \begin{pmatrix}
 \X{1}_2 & \X{1}_4 & \cdots & \X{1}_{2l} & 0 \\
 0   & \X{1}_2 & \cdots & \X{1}_{2l-2} & \X{1}_{2l}
\end{pmatrix}_{2k_b\times n_b(l+1)}.
\end{equation} 
Note here that
\[
 \rk\Bigl(\X{1}_1,\ldots,\X{1}_{2l+1}\Bigr)=k_a, \qquad
 \rk\Bigl(\X{1}_2,\ldots,\X{1}_{2l}\Bigr)=k_b,
\]
because
\[
 \begin{pmatrix} S \\ T \end{pmatrix}
 (X_1,\ldots,X_{2l+1},X_2,\ldots,X_{2l})
=\begin{pmatrix}
 \X{1}_1 & \cdots & \X{1}_{2l+1} &       * & \cdots & * \\
       0 & \cdots & 0            & \X{1}_2 & \cdots & \X{1}_{2l}
\end{pmatrix}
\]
is of row full rank $k=k_a+k_b$.

Let $\ra{1}$ be the rank of (\ref{eq:Xodd}), and let $a_1=\rk\bigl(\X{1}_1\bigr)$.  Then we obviously have
\begin{equation}
\label{eq:ineq_0}
 \ra{1} \ge a_1 + k_a.
\end{equation}
Take $A\in\mathrm{GL}(k_a,\mathbb{R})$ and $P\in\mathrm{GL}(n_a,\mathbb{R})$ such that
$A \X{1}_1 P=\begin{pmatrix} I_{a_1} \\ 0 \end{pmatrix}$; 
these are the first $k_a$ rows of $X_{1}$ in (\ref{eq:XX}).  Consider then transforming (\ref{eq:Xodd}) by multiplying the $2k_a\times 2k_a$ block-diagonal matrix $\diag(A,A)$ from the left and the $(l+2)n_a\times (l+2)n_a$ block-diagonal matrix $\diag(P,\ldots,P)$ from the right.  This transformation preserves rank and turns the matrix in  (\ref{eq:Xodd}) into
\begin{equation}
 \begin{pmatrix} \begin{pmatrix} I_{a_1} \\ 0 \end{pmatrix} & \tX_3 & \tX_5 & \cdots & \tX_{2l+1} & 0 \\
   0 & \begin{pmatrix} I_{a_1} \\ 0 \end{pmatrix} & \tX_3 & \cdots & \tX_{2l-1} & \tX_{2l+1} \end{pmatrix}_{2k_a\times n_a(l+2)},
\label{eq:XXX}
\end{equation}
where $\tX_{2i-1} = A \X{1}_{2i-1} P$.
For each $i$, let $\X{2}_{2i-1}$ be the submatrix consisting of the $(a_1+1)$st to $k_a$th row of $\tX_{2i-1}$.
Note that
\[
 \rk \Bigl(\X{1}_1,\ldots,\X{1}_{2l+1}\Bigr) = \rk \begin{pmatrix}
   I_{a_1} & *       & \cdots & * \\
   0       & \X{2}_3 & \cdots & \X{2}_{2l+1}
 \end{pmatrix}
\]
and hence
\begin{equation*}
 \rk \Bigl(\X{2}_3,\ldots,\X{2}_{2l+1}\Bigr) = \rk \Bigl(\X{1}_1,\ldots,\X{1}_{2l+1}\Bigr) -a_1 = k_a-a_1.
\end{equation*}
Deleting the rows with indices between  $k_a+1$ and $k_a+a_1$ from (\ref{eq:XXX}),
we have the inequality
\begin{align}
\ra{1}
&\ge \rk \begin{pmatrix} \begin{pmatrix} I_{a_1} \\ 0 \end{pmatrix} & \tX_3 & \tX_5 & \cdots & \tX_{2l+1} & 0 \\
 0 & 0 & \X{2}_3 & \cdots & \X{2}_{2l-1} & \X{2}_{2l+1} \end{pmatrix}_{(2k_a-a_1)\times n_a(l+2)} \nonumber \\
&= a_1 + \rk \begin{pmatrix} \X{2}_3 & \X{2}_5 & \cdots & \X{2}_{2l+1} & 0 \\
  0 & \X{2}_3 & \cdots & \X{2}_{2l-1} & \X{2}_{2l+1} \end{pmatrix}_{2(k_a-a_1)\times n_a(l+1)} \nonumber \\
&=: a_1 + \ra{2}.
\label{eq:ineq_a}
\end{align}

Repeating the procedure in the preceding paragraph $(i-1)$ times, we obtain the matrices $\X{i}_{2i-1},\ldots,\X{i}_{2l+1}$.
Let
\[
 a_i = \rk\bigl(\X{i}_{2i-1}\bigr)
\]
and
\[
 \ra{i} = \rk \begin{pmatrix} \X{i}_{2i-1} & \X{i}_{2i+1} & \cdots & \X{i}_{2l+1} & 0 \\
  0 & \X{i}_{2i-1} & \cdots & \X{i}_{2l-1} & \X{i}_{2l+1} \end{pmatrix}_{2\left(k_a-\sum_{j=1}^{i-1}a_j\right)\times n_a(l+3-i)}.
\]
Noting that
\begin{equation}
 \rk \Bigl(\X{i}_{2i-1},\ldots,\X{i}_{2l+1}\Bigr)
 = \rk \Bigl(\X{i-1}_{2i-3},\ldots,\X{i-1}_{2l+1}\Bigr) - a_{i-1}
 = k_a-\sum_{j=1}^{i-1} a_{j},
\label{eq:rankXi}
\end{equation}
we obtain that the inequalities in (\ref{eq:ineq_0}) and (\ref{eq:ineq_a}) extend to 
\begin{align}
\label{eq:ineq_0i}
 \ra{i} & \ge a_i + \biggl( k_a-\sum_{j=1}^{i-1}a_j \biggr), \\
\label{eq:ineq_ai}
 \ra{i} & \ge a_i + \ra{i+1},
\end{align}
respectively.
Here we let
$\ra{l+2}=0$.

From  (\ref{eq:ineq_0i}) and (\ref{eq:ineq_ai}), we find that 
\[
\ra{1} = \sum_{j=1}^{i-1} \bigl( \ra{j}-\ra{j+1} \bigr) + \ra{i}
\ge \sum_{j=1}^{i-1} a_j + a_i + \biggl( k_a-\sum_{j=1}^{i-1}a_j \biggr) = a_i + k_a
\]
for all $i$.
This is equivalent to
\begin{equation}
\label{eq:max_a}
 \ra{1} \ge \max_i a_i +k_a.
\end{equation}
From the construction (\ref{eq:rankXi}), we have
\begin{equation}
\label{eq:sum_a} 
 k_a = \sum_{i=1}^{l+1} a_i.
\end{equation}

Applying the same arguments, the rank $\rb{1}$ of (\ref{eq:Xeven}) is seen to satisfy the inequality
\begin{equation}
\label{eq:max_b}
 \rb{1} \ge \max_j b_j +k_b,
\end{equation}
where the $b_j=\rk\bigl(\X{j}_{2j}\bigr)$ are defined in analogy to the $a_i$ and satisfy the identity
\begin{equation}
\label{eq:sum_b}
 k_b = \sum_{j=1}^l b_j.
\end{equation}
Combining (\ref{eq:max_a}) and (\ref{eq:max_b}), the rank of the matrix (\ref{eq:XO}) can be bounded from below as
\begin{equation}
\label{eq:maxmax}
 \mbox{rank of }(\ref{eq:XO}) \ge \max_i a_i + \max_j b_j + k.
\end{equation}

The bound in \eqref{eq:maxmax} is for a given fixed matrix $X$ and in terms of the ranks $a_i$ and $b_j$ the matrix determines. To obtain a lower bound for all possible $X$, we may minimize the right-hand side of (\ref{eq:maxmax}) under the constraints the $a_i$ and $b_j$ should satisfy.  These constraints are given by (\ref{eq:sum_a}) and (\ref{eq:sum_b}).  We, thus, minimize over the set
\begin{multline*}
 \Bigl\{ \bigl(\max a_i,\max  b_j\bigr) \mid
 0\le a_i\le n_a, \ 0\le b_j\le n_b, \ %
 k = \textstyle \sum_i a_i + \sum_j b_j \Bigr\} \\
= \Bigl\{ \bigl(\max a_i,\max  b_j\bigr) \mid
  0\le \min a_i\le\max a_i\le n_a, \ 0\le \min b_j\le\max b_j\le n_b, \\ %
  \max a_i+\max b_j \le k \le \max a_i(l+1) + \max b_j l \Bigr\}.
\end{multline*}
As this set is contained in 
\begin{align*}
 \Bigl\{ \bigl(\max a_i,\max  b_j\bigr) \mid
  0\le \max a_i\le n_a, \ 0\le \max b_j\le n_b, \ %
  k \le \max a_i(l+1) + \max b_j l \Bigr\},
\end{align*}
we see that $r_2(m_1,m_2,k)$ is bounded from below by the right-hand side of (\ref{eq:r2}).
\end{proof}

\subsection{Evaluation of $S_2(m_1,m_2)$}

Our next goal is to evaluate the quantity
\begin{equation}
\label{eq:S2}
 S_2(m_1,m_2) = \min_{1\le k\le m_2-1} \bigl\{ m_2 r_2(m_1,m_2,k) - m_1 k \bigr\}
\end{equation}
whose sign determines the (in-)existence of the MLE.
As in the previous subsection, we refer to the numbers $l(m_1,m_2)$, $n_a$, and $n_b$ defined in (\ref{eq:l}) and (\ref{eq:nanb}), respectively.
Note first that
\[
 \frac{m_2}{m_1-m_2}\mbox{ is an integer} \ \iff \ n_b=0.
\]

\begin{theorem}
\label{thm:S2}
(i) If $m_1=m_2+1$, then 
$n_a=1$, $n_b=0$, and
\[
 S_2(m_1,m_2)=1,
\]
with the minimum in \eqref{eq:S2} attained iff $k=m_2-1$.

(ii) If $m_1>m_2+1$ and $\frac{m_2}{m_1-m_2}$ is an integer, then 
$n_a\ge 2$, $n_b=0$, and
\[
 S_2(m_1,m_2)=0,
\]
with the minimum in \eqref{eq:S2} attained iff $k$ is an integer multiple of $\frac{m_2}{m_1-m_2}$.

(iii)  If $m_1>m_2+1$ and $\frac{m_2}{m_1-m_2}$ is not an integer, then
$n_a\ge 1$, $n_b\ge 1$, and
\[
 S_2(m_1,m_2)=-n_a n_b,
\]
with the minimum in \eqref{eq:S2} attained iff $k=(l+1)n_a$.
\end{theorem}

\begin{proof}
(i) When $m_1=m_2+1$, we have $l=m_2-1$, $n_a=1$, and $n_b=0$.  It follows from Theorem~\ref{thm:r2} that $r_2(m_1,m_2,k)=k+1$.  We obtain that
\[
 S_2(m_1,m_2) = \min_{1\le k\le m_2-1} \bigl\{ m_2 (k+1) - (m_2+1) k \bigr\}
 = \min_{1\le k\le m_2-1} ( m_2 - k ) =1.
\]

\noindent
(ii)  When $m_1>m_2+1$ and $\frac{m_2}{m_1-m_2}$ is an integer, it holds that $l=\frac{m_2}{m_1-m_2}-1$,  $n_a=m_1-m_2$ and $n_b=0$.  Theorem~\ref{thm:r2} yields that
\[
 r_2(m_1,m_2,k)
 = k + \Bigl\lceil k \frac{m_1-m_2}{m_2} \Bigr\rceil.
\]
Consequently,
\[
 S_2(m_1,m_2) = \min_{1\le k\le m_2-1} m_2 \Bigl( \Bigl\lceil k \frac{m_1-m_2}{m_2} \Bigr\rceil - k \frac{m_1-m_2}{m_2} \Bigr) \ge 0.
\]
The lower bound $0$ is attained when $k \frac{m_1-m_2}{m_2}$ is an integer.

(iii) Finally, consider the case where $\frac{m_2}{m_1-m_2}$ is 
not an integer (trivially $m_1-m_2>1$).  Then $S_2(m_1,m_2)$ equals the minimum of the function
\begin{equation}
\label{eq:hkab}
 h(k,a,b)\;=\;m_2(a+b+k)-k m_1 \;=\;m_2(a+b)-(m_1-m_2)k
\end{equation}
over the set
\begin{align*}
 \bigl\{ (k,a,b) \mid 1\le k\le m_2-1,\ k\le a (l+1) + b l,\ 0\le a\le n_a,\ 0\le b\le n_b\bigr\}.
\end{align*}
We distinguish two cases of how the minimum may be attained, namely, case 1 with $m_2-1\ge a (l+1) + b l$, and case 2 with $m_2-1< a (l+1) + b l$.  Accordingly,
\[
 S_2(m_1,m_2) = \min\{S_2'(m_1,m_2),S_2''(m_1,m_2)\},
\]
where $S_2'(m_1,m_2)$ and $S_2''(m_1,m_2)$ are the minima of $h(k,a,b)$ from~\eqref{eq:hkab} over the sets 
\begin{align*}
 \bigl\{ (k,a,b) \mid 1\le k\le a (l+1) + b l
\le m_2-1,\ 0\le a\le n_a,\ 0\le b\le n_b\bigr\}
\end{align*}
and
\begin{align}
\label{eq:min-domain}
\bigl\{ (k,a,b) \mid 1\le k\le m_2-1< a (l+1) + b l,\ 0\le a\le n_a,\ 0\le b\le n_b\bigr\},
\end{align}
respectively.

Case 1:  The minimum $S_2'(m_1,m_2)$ is attained iff $k=a(l+1)+b l$, in which case
\[
h(k,a,b) =m_2(a+b)-(m_1-m_2)\{a (l+1) + b l\} =-n_b a + n_a b.
\]
Therefore,
\begin{align*}
 S_2'(m_1,m_2)
=& \min\bigl\{ -n_b a + n_a b \mid a (l+1) + b l\le m_2-1,\ 0\le a\le n_a,\ 0\le b\le n_b\bigr\}. 
\end{align*}
This minimum is achieved by taking $b$ as small as possible, so $b=0$, and $a$ is large as possible.  Indeed, $(a,b)=(n_a,0)$ is feasible as
\[
(m_2-1)-[n_a(l+1)+0\cdot l]  = l n_b-1 \ge 0.
\]
We conclude that $S_2'(m_1,m_2)=-n_a n_b$, with the minimum attained at $k=n_a(l+1)$.

Case 2: Because
\[
 m_2-1 < a(l+1)+b l\le n_a(l+1)+n_b l=m_2,
\]
the set (\ref{eq:min-domain}) is
\[
 \{ (k,a,b) \mid 1\le k\le m_2-1,\,a=n_a,\,b=n_b \}.
\]
The minimum $S_2''(m_1,m_2)$ is attained iff $k=m_2-1$, in which case
\[
h(k,a,b) = m_2(n_a+n_b)-(m_1-m_2)(m_2-1) = m_1-m_2>0.
\]

In summary, $S_2(m_1,m_2) = \min\{S_2'(m_1,m_2),S_2''(m_1,m_2)\}=-n_a n_b$, and this minimum is attained iff $k=n_a(l+1)+0=n_a(l+1)$.
\end{proof}
\begin{remark}
When $S_2(m_1,m_2)=0$, neither $m_1$ nor $m_2$ is a prime number.
\end{remark}

Including the square case, the possible values of $S_2(m_1,m_2)$ may be summarized as follows.
The values are tabulated up to $m_1\le 17$ in Table \ref{table:S2}.
\begin{proposition}
For $n=2$ generic data matrices $Y_1,Y_2\in\mathbb{R}^{m_1\times m_2}$ with $m_2\le m_1< 2m_2$,
\begin{align*}
S_2(m_1,m_2)
&= \begin{cases}
 0 & (m_1=m_2\ge 3 \mbox{ or } \\
   & \ m_1=m_2=2 \mbox{ and $Y_1^{-1}Y_2$ has real eigenvalues}), \\
 2 & (m_1=m_2=2 \mbox{ and $Y_1^{-1}Y_2$ has complex eigenvalues}), \\
 1 & (m_1=m_2+1), \\
 0 & (m_1>m_2+1,\,m_1-m_2\vert m_2), \\
 -n_a n_b & (m_1-m_2\!\not\vert m_2).
   \end{cases}
\end{align*}
\end{proposition}

We recall that
the MLE exists uniquely if $S_2(m_1,m_2)>0$,
exists non-uniquely if $S_2(m_1,m_2)=0$, and 
does not exist if $S_2(m_1,m_2)<0$.


\begin{table}[h]
\caption{$S_2(m_1,m_2)$ ($m_1\ge m_2$).}
\label{table:S2}
\begin{center}
\begin{footnotesize}
\begin{tabular}{c||ccccccccccccccc}
\hline
\small\diagbox[width=7ex]{$m_1$}{$m_2$}
  & 2 & 3 & 4 & 5 & 6 & 7 & 8 & 9 & 10 & 11 & 12 & 13 & 14 & 15 & 16 \\
\hline\hline
2 & \z{0/2} \\
3 & \p{1} & \z{0} \\
4 & \z{0} & \p{1} & \z{0} \\
5 & \m{} & \m{--1} & \p{1} & \z{0} \\
6 & \m{} & \z{0} & \z{0} & \p{1} & \z{0} \\
7 & \m{} & \m{} & \m{--2} & \m{--1} & \p{1} & \z{0} \\
8 & \m{} & \m{} & \z{0} & \m{--2} & \z{0} & \p{1} & \z{0} \\
9 & \m{} & \m{} & \m{} & \m{--3} & \z{0} & \m{--1} & \p{1} & \z{0} \\
10 & \m{} & \m{} & \m{} & \z{0} & \m{--4} & \m{--2} & \z{0} & \p{1} & \z{0} \\
11 & \m{} & \m{} & \m{} & \m{} & \m{--4} & \m{--3} & \m{--2} & \m{--1} & \p{1} & \z{0} \\
12 & \m{} & \m{} & \m{} & \m{} & \z{0} & \m{--6} & \z{0} & \z{0} & \z{0} & \p{1} & \z{0} \\
13 & \m{} & \m{} & \m{} & \m{} & \m{} & \m{--5} & \m{--6} & \m{--3} & \m{--2} & \m{--1} & \p{1} & \z{0} \\
14 & \m{} & \m{} & \m{} & \m{} & \m{} & \z{0} & \m{--8} & \m{--4} & \m{--4} & \m{--2} & \z{0} & \p{1} & \z{0} \\
15 & \m{} & \m{} & \m{} & \m{} & \m{} & \m{} & \m{--6} & \m{--9} & \z{0} & \m{--3} & \z{0} & \m{--1} & \p{1} & \z{0} \\
16 & \m{} & \m{} & \m{} & \m{} & \m{} & \m{} & \z{0} & \m{--10} & \m{--8} & \m{--4} & \z{0} & \m{--2} & \z{0} & \p{1} & \z{0} \\
17 & \m{} & \m{} & \m{} & \m{} & \m{} & \m{} & \m{} & \m{--7} & \m{--12} & \m{--5} & \m{--6} & \m{--3} & \m{--2} & \m{--1} & \p{1} \\
\hline
\end{tabular}
\end{footnotesize}
\end{center}
\smallskip

\begin{flushleft}
\begin{tiny}
\begin{tabular}{|c|}
\hline \p{} \\ \hline 
\end{tabular}
\end{tiny}{\scriptsize : MLE exists uniquely.}

\begin{tiny}
\begin{tabular}{|c|}
\hline \z{} \\ \hline 
\end{tabular}
\end{tiny}{\scriptsize : MLE exists non-uniquely.}

\begin{tiny}
\begin{tabular}{|c|}
\hline \m{} \\ \hline 
\end{tabular}
\end{tiny}{\scriptsize : MLE does not exist.}
\end{flushleft}
\end{table}

\section{When the column size of matrices is 2}

When the matrix size of $Y_i$ is $m_1\times 2$, that is, $m_2=2$,
we can find $S_n(m_1,2)$ as a byproduct of the case $n=2$ discussed in Sections~\ref{sec:square-matrices} and \ref{sec:rectangular-matrices}.
Indeed, from the definitions (\ref{eq:minimal-rank-summary}) and (\ref{eq:minimal-rank}),
\[
S_n(m_1,2) = \min_{1\le k<2} \bigl\{ 2 r_n(m_1,2,k)-m_1 k \bigr\} = 2 r_n(m_1,2,1)-m_1,
\]
and
\begin{align*}
    r_n(m_1,2,1)
 &= \min_{X\in\mathbb{R}^{2\times 1}: \rk(X)=1}\rk(Y_1 X,\ldots,Y_n X) \\
 &= \min_{(x_1,x_2)\in\mathbb{R}^2\setminus\{0\}} \rk\bigl(x_1 Y_{(1)} + x_2 Y_{(2)}\bigr),
\end{align*}
where $Y_{(j)}=(y_{1j},\ldots,y_{nj})$, $j=1,2$, with $y_{ij}$ the $j$th column vector of $Y_i$.
The matrices $Y_{(j)}$ are $m_1\times n$ generic matrices.
We examine the cases (i) $m_1=n$ and (ii) $m_1\ne n$ separately.

(i) Let $m=m_1=n$. Since the $Y_{(j)}$ are non-singular, we may define $W=Y_{(1)}^{-1}Y_{(2)}$.
Then, $r_n(m_1,2,1)=\min_{x\in\mathbb{R}}\rk(x I_m+W)=m-1$ if $W$ has a real eigenvalue, and $r_n(m_1,2,1)=m$ otherwise.
Therefore,
$S_n(m,2) = 2 r_m(m,2,1) - m = m-2$ if $W$ has a real eigenvalue, and $S_n(m,2)=m$ otherwise.

(ii) Suppose that $m_1>n$.
Then, by multiplying a suitable $(A,C)\in \mathrm{GL}(m_1)\times\mathrm{GL}(n)$ from left and right to get the Kronecker canonical form, we have
\begin{equation}
\label{kronecker_again}
 r_n(m_1,2,1) = \min_{(x_1,x_2)\ne 0}\rk\biggl(x_1 \begin{pmatrix}I_n \\ 0 \end{pmatrix} + x_2\begin{pmatrix}0 \\ I_n \end{pmatrix}\biggr) =n.
\end{equation}
Hence, $S_n(m_1,2) = 2n-m_1$.
Similarly, when $m_1<n$, we have $r_n(m_1,2,1)=m_1$ and $S_n(m_1,2)=m_1$.
Recall that we are examining the region $m_2\le m_1<n m_2$, or equivalently $1\le m_1/2<n$ when $m_2=2$.
\begin{proposition}
\label{prop:m2=2}
For generic data matrices $Y_1,\ldots,Y_n\in\mathbb{R}^{m_1\times 2}$ with $1\le m_1/2<n$,
\begin{align*}
 S_n(m_1,2)
 =& \begin{cases}
 2 n-m_1 & (n<m_1), \\
 m_1 & (m_1<n), \\
 m & (m_1=n=m,\,\mbox{$W$ does not have real eigenvalues}), \\
 m-2 & (m_1=n=m,\,\mbox{$W$ has a real eigenvalue}).
 \end{cases}
\end{align*} 
That is, (i) $S_n(m_1,2)>0$ when $m_1\ne n$, or when $m_1=n>2$, or when $m_1=n=2$ and $W$ does not have real eigenvalues; (ii) $S_n(m_1,2)=0$ when $m_1=n=2$ and $W$ has a real eigenvalue.
\end{proposition}

\begin{remark}
\label{rem:third_invariance}
The transformation from $Y_i$'s to the Kronecker form in (\ref{kronecker_again}) is written as $(Y_1,\ldots,Y_n)\mapsto A(Y_1,\ldots,Y_n)(C\otimes I_{m_2})$, where $C$ is an $n\times n$ non-singular matrix.
Here we used the fact that $\rk(Y_1,\ldots,Y_n)(C\otimes I_n)$ is invariant as long as $C\in\mathrm{GL}(n)$.
Combined with the group action discussed in Section~\ref{sec:group-action}, the group action
\[
 (Y_1,\ldots,Y_n) \mapsto A (Y_1,\ldots,Y_n)(C\otimes B), \quad (A,B,C)\in\mathrm{GL}(m_1)\times\mathrm{GL}(m_2)\times\mathrm{GL}(n)
\]
keeps the values $r_n(m_1,m_2,k)$, and hence also $S_n(m_1,m_2)$, invariant.
\end{remark}

\section{Maximum likelihood estimation for two data matrices}
\label{sec:mle}

In this section we derive the precise form of maximizers of the likelihood function for $n=2$ rectangular data matrices of size $m_1\times m_2$ with $2m_2\ge m_1>m_2$.  We first give the MLE when it exists uniquely.  This then allows us to show existence of maximizers in the cases where the study of ranks in Section~\ref{sec:rectangular-matrices} implies that the likelihood function is bounded.  

\subsection{Closed form MLEs for $m_1=m_2+1$}
\label{sec:closed-form-mle-m2+1}

Consider the case where $m_1=m_2+1$ and $n=2$, so that a unique MLE exists almost surely.
For simpler notation, let $m:=m_2$.
By Theorem~\ref{thm:tenberge}, we may assume that the two data matrices are
\begin{align}
  Y_1 =
        \begin{pmatrix}
          I_{m} \\
          0_{1,m}
        \end{pmatrix}, \qquad
  Y_2 =
        \begin{pmatrix}
          0_{1,m} \\
          I_{m}
        \end{pmatrix}.
\end{align}
Then the negated profile log-likelihood function takes the form
\begin{align*}
  g_0(\Phi) &= m\log\det\left( 
            \begin{pmatrix}
              \Phi & 0 \\
              0    & 0
            \end{pmatrix}  +
            \begin{pmatrix}
              0 & 0 \\
              0 & \Phi
            \end{pmatrix}
\right) - (m+1) \log\det(\Phi).
\end{align*}

\begin{proposition}
\label{prop:closed-form-mle-m2+1}
  When restricted to matrices $\Phi=(\phi_{jk})$ with $\phi_{11}=1$,
  the function $g_0$ is uniquely minimized by the diagonal matrix
  \[
    \Phi_0 \;=\; \diag\left( \binom{m-1}{j-1} : j=1,\dots,m\right).
  \]
\end{proposition}
\begin{proof}
    The existence of a unique minimizer is clear from Theorem~\ref{thm:S2}.  It thus 
     suffices to show that $\Phi_0$ is a
  critical point of $g_0$.

  The logarithm of
  the determinant has differential
  \[
    \mathrm{d}\log\det(\Phi) = \tr\left( \Phi^{-1} \mathrm{d}\Phi\right).
  \]
  It follows that the differential of $g_0$ is
  \begin{multline*}
    \mathrm{d}g_0(\Phi;U) 
    = \\
    m\tr\left\{\left[ 
      \begin{pmatrix}
        \Phi & 0 \\
        0    & 0
      \end{pmatrix}  +
            \begin{pmatrix}
              0 & 0 \\
              0 & \Phi
            \end{pmatrix}
    \right]^{-1} \left[\begin{pmatrix}
      U & 0\\
      0 & 0
    \end{pmatrix}  +
          \begin{pmatrix}
            0 & 0\\
            0 & U
          \end{pmatrix}\right]
    \right\} - (m+1) \tr\left( \Phi^{-1}U\right).
  \end{multline*}
  Since our candidate $\Phi_0$ is diagonal,
  \begin{equation*}
    \mathrm{d}g_0(\Phi_0;U) =\\
    m \left( \frac{u_{11}}{\phi_1} + \frac{u_{mm}}{\phi_m}+
      \sum_{j=1}^{m-1}
      \frac{u_{jj}+u_{j+1,j+1}}{\phi_j+\phi_{j+1}}\right)
    - (m+1) \sum_{j=1}^m \frac{u_{jj}}{\phi_j},
  \end{equation*}
  where $\phi_j$ is the $j$-th diagonal entry of $\Phi_0$, and
  $U=(u_{jk})$.  The differential $\mathrm{d}g_0(\Phi_0;U)$ is zero if
  \begin{align}
    \label{eq:mle:1}
    \frac{m+1}{\phi_1} 
    &= m\left(\frac{1}{\phi_1}+\frac{1}{\phi_1+\phi_2}\right), \\
    \label{eq:mle:2}
    \frac{m+1}{\phi_j} 
    &=
      m\left(\frac{1}{\phi_j+\phi_{j-1}}+\frac{1}{\phi_j+\phi_{j+1}}\right),
      \quad j=2,\dots,m-1,\\
    \label{eq:mle:3}
    \frac{m+1}{\phi_m} 
    &= m\left(\frac{1}{\phi_m}+\frac{1}{\phi_m+\phi_{m-1}}\right).
  \end{align}
  
  It is easy to see that the first equation, that in~(\ref{eq:mle:1}),
  holds for our choice of $\Phi_0$.  Indeed, after clearing
  denominators, the equation becomes
  \[
    (m+1)\left(\phi_1+\phi_2 \right) \;=\; m\left(
      2\phi_1+\phi_2\right)  \iff \phi_2 = (m-1)\phi_1,
  \]
  which clearly holds when $\phi_1=\binom{m-1}{0}=1$ and
  $\phi_2=\binom{m-1}{1}=m-1$.   The last equation in~(\ref{eq:mle:3})
  holds similarly as we have $\phi_m=\binom{m-1}{m-1}=1$ and
  $\phi_{m-1}=\binom{m-1}{m-2}=m-1$.   

  Let $2\le j\le m-1$.  Solving the $j$-th equation
  in~(\ref{eq:mle:2}) for $\phi_{j+1}$ we obtain
  \begin{equation}
    \label{eq:MLE:2a}
    \phi_{j+1} \;=\; \phi_j\cdot\frac{
        (m-1)\phi_j-\phi_{j-1}}{\phi_j + (m+1)\phi_{j-1}}. 
  \end{equation}
  Using that
  \[
    \binom{m-1}{j-1} = \frac{m-j+1}{j-1} \binom{m-1}{j-2},
  \]
  we derive that
  \[
    \frac{
      (m-1)\phi_j-\phi_{j-1}}{\phi_j + (m+1)\phi_{j-1}} =
    \frac{
      (m-1)\binom{m-1}{j-1}-\binom{m-1}{j-2}}{\binom{m-1}{j-1} +
      (m+1)\binom{m-1}{j-2}} = 
    \frac{
      (m-1) \frac{m-j+1}{j-1}-1}{ \frac{m-j+1}{j-1} +
      (m+1)} = \frac{m-j}{j}.  
    \]
  We see that~(\ref{eq:MLE:2a}) holds because
  \[
    \phi_{j+1} = \binom{m-1}{j} = \frac{m-j}{j}\binom{m-1}{j-1} =
    \frac{m-j}{j}\phi_j. 
  \]
  We have thus shown that our choice of $\Phi_0$ satisfies
  $\mathrm{d}g_0(\Phi_0;U)\equiv 0$.  
\end{proof}

Simple calculations yield that, at the critical point,
\begin{equation}
 g_0(\Phi_0) = m \log d(m) - (m+1) \log e(m),
\end{equation}
where
\begin{equation*}
 d(m)
= \det\left(\sum_{i=1}^2 Y_i \Phi_0^{-1} Y_i^T\right)
= \Bigl(\frac{m}{m-1}\Bigr)^{m-1} \frac{1}{\prod_{j=1}^{m-1}\binom{m-2}{j-1}} 
\end{equation*}
and
\begin{equation*}
e(m)
= \det\Psi(m)
= \frac{1}{\prod_{j=1}^{m}\binom{m-1}{j-1}}.
\end{equation*}

\subsection{Critical points when MLEs exist non-uniquely}
\label{sec:closed-form-mle-nonunique}

Non-unique existence corresponds to case (ii) in Theorem~\ref{thm:S2}.  So, $n_b=0$ and $n_a=m_1-m_2\ge 2$.
Moreover, $m_1=(l+2)n_a$ and $m_2=(l+1)n_a$.
Assume, as before, that the two $m_1\times m_2$ data matrices are
\begin{align}
  Y_1 =
        \begin{pmatrix}
          I_{m_2} \\
          0_{m_1-m_2,m_2}
        \end{pmatrix}, \qquad
  Y_2 =
        \begin{pmatrix}
          0_{m_1-m_2,m_2} \\
          I_{m_2}
        \end{pmatrix}.
\end{align}
As seen in Remark~\ref{remark:kroneckerform}, by permuting the rows and the columns of $Y_1$ and $Y_2$ simultanously, we may transform $Y_1$ and $Y_2$ into
\[
 Y_1 = \begin{pmatrix}
   U_{l+1} &        & 0 \\
           & \ddots &   \\
   0       &        & U_{l+1}
   \end{pmatrix}, \qquad
 U_{l+1} = \begin{pmatrix}
   I_{l+1} \\
   0_{1,l+1}
  \end{pmatrix} \in\mathbb{R}^{(l+2)\times(l+1)}
\]
and
\[
 Y_2 = \begin{pmatrix}
   L_{l+1} &        & 0 \\
           & \ddots & \\
   0       &        & L_{l+1}
       \end{pmatrix}, \qquad
 L_{l+1} = \begin{pmatrix}
   0_{1,l+1} \\
   I_{l+1}
   \end{pmatrix} \in\mathbb{R}^{(l+2)\times(l+1)}.
\]
Let $\Phi=\diag(\Phi_1,\ldots,\Phi_{n_a})$ with $\Phi_j\in\mathbb{R}^{(l+1)\times(l+1)}$.
Then the negated profile log-likelihood function takes the form
\begin{align*}
  g_0(\Phi) &= \sum_{j=1}^{n_a}\left[ m_2 \log\det\left( 
            \begin{pmatrix}
              \Phi_j & 0 \\
                   0 & 0
            \end{pmatrix} +
            \begin{pmatrix}
              0 & 0 \\
              0 & \Phi_j \\
            \end{pmatrix}
\right) - m_1 \log\det(\Phi_j) \right].
\end{align*}
Applying Proposition \ref{prop:closed-form-mle-m2+1} to each summand we may determine critical points $\Phi_{0j}$ for each summand in $g_0(\Phi)$.  These may then be combined to obtain critical points of $g_0$.

\begin{proposition}
\label{prop:crit-point-nonunique}
  When restricted to matrices $\Phi=(\phi_{jk})$ with $\phi_{11}=1$,
  the function $g_0$ has a critical point at every diagonal matrix
  $\Phi_0=(\phi^0_{jk})$ whose diagonal entries are
  \[
    \phi^0_{(j-1)n_a+k,(j-1)n_a+k}
 \;=\; c_j \binom{\ell}{k-1}, \qquad
    j=1,\dots,n_a,\quad k=1,\dots,\ell+1,
  \]
  where $c_1=1$ and $c_j>0$ arbitrary for $j=2,\dots,n_a$.
\end{proposition}

Note that for the critical points $\Phi_0$ from Proposition~\ref{prop:crit-point-nonunique} it holds that
\begin{align*}
 g_0(\Phi_0)
=& \sum_{j=1}^{n_a} \Bigl[ m_2 \log c_j^{l+2} d(l+1) - m_1 \log c_j^{l+1} e(l+1) \Bigr] \\
=& \sum_{j=1}^{n_a} \Bigl[ m_2 \log d(l+1) - m_1 \log e(l+1) \Bigr],
\end{align*}
which is independent of $c_j$'s.
This is a confirmation of the fact that all of the critical points define minima of $g_0$.

\section{Conclusion}
\label{sec:conclusion}

In this paper we considered uniqueness and existence of the maximum likelihood estimator in the matrix normal model.  In other words, we considered Gaussian models for i.i.d.~matrix-valued observations $Y_1,\dots,Y_n$ that posit a Kronecker product for the joint covariance matrix of the entries of each random matrix $Y_i$.  Our goal was give precise formulas for maximum likelihood thresholds, which are defined to be the sample sizes that are minimally needed for almost sure~existence of the MLE, unique existence of the MLE, or possibly mere boundedness of the likelihood function.  Our main result solves this problem for data matrices whose dimensions $m_1$ and $m_2$ differ at most by a factor of two.  Our solution exhibits subtle dependencies on $m_1$ and $m_2$.  From a statistical perspective our work clarifies that very small sample sizes are sufficient to make matrix normal models amenable to likelihood inference.  

As observed in the introduction, prior work of \cite{soloveychik:2016} can be used to determine the maximum likelihood thresholds for settings where one matrix dimension is sufficiently large compared to the other or, more precisely put, where dividing one matrix dimension by the other leaves a sufficiently small remainder.  In intermediate settings, however,  the maximum likelihood thresholds remain unknown.  Although good bounds exist, it would be of obvious interest to determine the thresholds in full generality.  Here it should be noted that our solution for the setting $2m_2\ge m_1\ge m_2$ crucially relies on invariance properties that allowed us to exploit the Kroncker canonical form for matrix pencils.  For larger sample sizes, new additional ideas are needed as a similarly simple canonical form does not exist \cite[Chap.~10]{landsberg}.  

In all cases covered by our results almost sure boundedness of the likelihood function implies almost sure existence of a maximizer.  We conjecture this to be true in general.  This said, there do exist individual data sets for which the likelihood function is bounded but does not achieve its maximum; recall Proposition~\ref{prop:square-2x2}.

Finally, we would like to note that since our paper was submitted a new preprint was posted on arXiv, which makes a connection between maximum likelihood estimation and computational invariant theory for a series of problems including the matrix normal models considered here \cite{amendola}.  A further preprint appeared later announcing that a full solution to the (unique) existence problem can be achieved via quiver representation theory \cite{makam}.

\section*{Acknowledgements}
We are grateful to Satoru Iwata, Lek-Heng Lim and Fumihiro Sato for their comments on the Kronecker canonical form.  SK was partially supported by JSPS KAKENHI Grant Number JP16H02792.

\begin{appendix}

\section{Some lemmas and proofs}

\begin{proof}[Proof of Lemma~\ref{lem:Y-full-row-rank}]
  The $m_1\times m_1$ matrix $\sum_{i=1}^n Y_iY_i^T$ is positive
  semidefinite of rank at most $n m_2<m_1$.  By spectral
  decomposition, 
  \[
    \sum_{i=1}^n Y_i Y_i^T = Q^T D Q
  \]
  where $Q=Q(Y)$ is $m_1\times m_1$ orthogonal, and $D$ is diagonal
  with $D_{m_1m_1}=0$.  Let $I_{m}$ be the $m\times m$ identity
  matrix, and let $e_m=(0,\dots,0,1)^T$ the $m$-th canonical basis
  vector of $\mathbb{R}^{m}$.
  Define
  $\Psi_1^{(t)} = Q(I_{m_1}+t\cdot e_{m_1}e_{m_1}^T )
  Q^T\in\mathit{PD}(m_1)$.  Then
  \begin{align*}
    \ell(\Psi_1^{(t)},I_{m_2}) 
    &= n m_2\log\det(I_{m_1}+t\cdot e_{m_2}e_{m_2}^T) -
    \tr[(I_{m_1}+t\cdot e_{m_2}e_{m_2}^T) D] \\
    &= n m_2\log(1+t) -
    \tr(D)
  \end{align*}
  tends to $\infty$ as $t\to\infty$.
\end{proof}

\begin{proof}[Proof of Theorem~\ref{thm:divisible}.]
We are assuming that $m_1/m_2$ is an integer.  Take $n=m_1/m_2$.
Then $Y=(Y_1,\ldots,Y_n)$ is $n\times n$,
and the profile log-likelihood function $g(\Psi)$ in (\ref{eq:profile-neglik-Y}) is easily seen to be constant.  Indeed,
\begin{align*}
 g(\Psi)
 &= m_2 \log \det(Y)^2 + m_2 \log \det(\Psi)^n - m_1 \log \det(\Psi) 
 = 2 m_2 \log|\det(Y)|.
\end{align*}
Hence, the function is maximized by any matrix $\Psi\in \mathit{PD}(m_2)$, and we obtain that $N_b(m_1,m_2)=N_e(m_1,m_2)\le m_1/m_2$ and $N_u(m_1,m_2)>m_1/m_2$.  The lower bound in Proposition \ref{prop:known-results} now implies that $N_b(m_1,m_2)=N_e(m_1,m_2)=m_1/m_2$.

When $m_1=m_2$, Corollary \ref{cor:square-Nu} gives $N_u(m_1,m_2)=3$.  When $m_1>m_2$, Proposition \ref{prop:known-results} gives the upper bound $N_u(m_1,m_2)\le m_1/m_2+1$, which implies $N_u(m_1,m_2)=m_1/m_2+1$.
\end{proof}

\begin{proof}[Proof of Corollary~\ref{cor:resolved-already}]
  Let $m_1=hm_2+r$ with quotient $h=\lfloor m_1/m_2\rfloor \ge 1$ and
  remainder $r =m_1\bmod m_2$.  Applying
  Proposition~\ref{prop:known-results},
  we have
  \[
    h+1\;\le\; N_b(m_1,m_2)\;\le\; N_e(m_1,m_2)\;\le\; N_u(m_1,m_2).
  \]
  The upper bound from Proposition~\ref{prop:known-results} implies
  that all thresholds are equal to $h+1$ if
  \[
    \frac{m_1}{m_2}+\frac{m_2}{m_1} < h+1 \;\iff\; m_1^2+m_2^2<(h+1)m_1m_2.
  \]
  Substituting $m_1=hm_2+r$ and simplifying, this condition is equivalent to 
  \[
    m_2(m_2-r)h-(m_2^2-m_2 r+r^2)\;>\;0,
  \]
  and so equivalent to the claimed inequality
  \[
    h \;>\; \frac{m_2^2-m_2 r+r^2}{m_2(m_2-r)}.
  \]
  
  The right-hand side of the inequality just given is increasing in
  the remainder $r$.  Thus, for fixed $m_2$, it never exceeds
  \[
    \frac{m_2^2-m_2 (m_2-1)+(m_2-1)^2}{m_2(m_2-m_2+1)} \;=\; m_2
    -\frac{m_2-1}{m_2} < m_2.
  \]
  Hence, $h\ge m_2$ is sufficient for all thresholds being equal to $h+1$.  
\end{proof}

\begin{proof}[Proof of Theorem~\ref{thm:main}]
Proposition \ref{prop:square-n1} gives $N_e(m_1,m_2) = N_b(m_1,m_2)$ for $m_1=m_2$.
Corollary \ref{cor:square-Nu} gives $N_u(m_1,m_2)$ for $m_1=m_2$.
Theorem \ref{thm:divisible} yields
$N_e(m_1,m_2) = N_b(m_1,m_2)$ and $N_u(m_1,m_2)$ when $2 m_2=m_1$.
When $2 m_2>m_1>m_2$, by Theorem~\ref{thm:r2}, we have that $N_u(m_1,m_2)=2$ if $m_1=m_2+1$, and $N_u(m_1,m_2)>2$ otherwise,
and that $N_b(m_1,m_2)=N_e(m_1,m_2)=2$ if $m_1-m_2\vert m_2$, and $N_u(m_1,m_2)>2$ otherwise.
On the other hand, by Proposition \ref{prop:known-results}, $N_b(m_1,m_2)\le N_e(m_1,m_2)\le N_u(m_1,m_2)\le \lfloor m_1/m_2+m_2/m_1\rfloor+1\le 3$.
\end{proof}

\begin{lemma}
\label{lem:classify}
Let $x_i^{(t)}>0$, $1\le i\le m$, $t=1,2,\ldots$, be positive sequences.
After suitable relabeling of the indices $i$ of $x_i^{(t)}$, we can take a subsequence of the form
\begin{equation}
\label{eq:form}
 \begin{pmatrix} x_1^{(t)} \\ \vdots \\ x_{r_1}^{(t)} \\ x_{r_1+1}^{(t)} \\ \vdots \\ x_{r_1+r_2}^{(t)} \\ \vdots \\ x_{m-r_K+1}^{(t)} \\ \vdots \\ x_m^{(t)} \end{pmatrix} =
 \epsilon_1^{(t)} \begin{pmatrix} y_{11}^{(t)} \\ \vdots \\ y_{1 r_1}^{(t)} \\ 0 \\ \vdots \\ 0 \\ \vdots \\ 0 \\ \vdots \\ 0 \end{pmatrix} +
 \epsilon_2^{(t)} \begin{pmatrix} 0 \\ \vdots \\ 0 \\ y_{21}^{(t)} \\ \vdots \\ y_{2 r_2}^{(t)} \\ \vdots \\ 0 \\ \vdots \\ 0 \end{pmatrix} + \cdots +
 \epsilon_K^{(t)} \begin{pmatrix} 0 \\ \vdots \\ 0 \\ 0 \\ \vdots \\ 0 \\ \vdots \\ y_{K1}^{(t)} \\ \vdots \\ y_{K r_K}^{(t)} \end{pmatrix}
\end{equation}
where $r_1+\cdots+r_K=m$, each sequence $y_{ij}^{(t)}$ converges to a limit $y^0_{ij}>0$, and $\epsilon_{i+1}^{(t)}/\epsilon_i^{(t)}\to 0$, $1\le i\le K-1$, as $t\to\infty$.
\end{lemma}

\begin{proof}
When comparing two sequences $x_i^{(t)}$ and $x_j^{(t)}$, at least one of the two statements
``$x_i^{(t)}\ge x_j^{(t)}$ for infinitely many $t$'' or
``$x_i^{(t)}\le x_j^{(t)}$ for infinitely many $t$'' holds.
By relabeling the indices, we can assume that $x_1^{(t)}\ge x_i^{(t)}$ ($1<i\le m$) for infinitely many $t$.
Take subsequences so that $x_i^{(t)}/x_1^{(t)}\in(0,1]$ have finite limits for all $i$.
Suppose that the limits for $i=2,\ldots,r_1$ are positive, and the others are zero. 
Then we set $\epsilon_1^{(t)}=x_1^{(t)}$, $y_{1i}^{(t)}=x_i^{(t)}/x_1^{(t)}$, $i=1,\ldots,r_1$.

Next, we apply the same procedure to $x_{r_1+1}^{(t)},\ldots,x_m^{(t)}$.
Suppose that $x_{r_1+1}^{(t)}\ge x_i^{(t)}$ ($r_1+1<i\le m$) for infinitely many $t$.
Take subsequences so that $x_i^{(t)}/x_{r_1+1}^{(t)}\in(0,1]$ have finite limits for all $r_1+1<i\le m$ again, and classify the limits by their signs (i.e., positive or zero).
Suppose that the limits for $i=r_1+2,\ldots,r_1+r_2$ are positive, and the others are zero.
Then we take $\epsilon_2^{(t)}=x_{r_1+1}^{(t)}$ and $y_{2i}^{(t)}=x_{r_1+i}^{(t)}/x_{r_1+1}^{(t)}$, $i=r_1+1,\ldots,r_1+r_2$.
Note that $\epsilon_2^{(t)}/\epsilon_1^{(t)}=x_{r_1+1}^{(t)}/x_1^{(t)}\to 0$.

By repeating this procedure, we get the form (\ref{eq:form}).
\end{proof}

\begin{lemma}
  \label{lem:ev-prob}
  Let $Y_1$ and $Y_2$ be two independent random $2\times 2$ matrices
  whose entries are i.i.d.~$\mathcal N(0,1)$.  Then the matrix $Y_1^{-1}Y_2$
  has real eigenvalues with probability $\pi/4\approx 0.7854$.
\end{lemma}
\begin{proof}
  We first obtain the distribution of $Z=Y_1^{-1}Y_2$, which we do for
  general matrix size $m$.  We use $|A|$ as a shorthand for the
  determinant of a matrix $A$.  Then noting that $Y_2=Y_1 Z$ and
  $dY_2=|Y_1|^m dZ$, the joint density of $(Y_1,Y_2)$ is
\begin{align*}
& \tfrac{1}{(2\pi)^{m^2}}\exp\left\{-\tfrac{1}{2}\tr\left(Y_1 Y_1^T+Y_2 Y_2^T\right)\right\} dY_1 dY_2 \\
&= \tfrac{1}{(2\pi)^{m^2}}\exp\left\{-\tfrac{1}{2}\tr\left(Y_1 Y_1^T+Y_1 Z Z^T Y_1^T\right)\right\} dY_1 |Y_1|^m dZ \\
&= \tfrac{1}{(2\pi)^{m^2}}\exp\left\{-\tfrac{1}{2}\tr\left( Y_1^T Y_1 (I + Z Z^T)\right)\right\} |Y_1|^m dY_1 dZ.
\end{align*}
Let $Y_1 = H T$ be the QR decomposition.  That is, $H\in O(m)$ and
$T=(t_{ij})$ is an upper-triangular matrix with $t_{ii}>0$.  Let
$S=Y_1^TY_1=T^TT$.  Then the mapping $Y_1\mapsto (S,H)$ is one to one,
and according to Theorem 2.1.14 in \cite[p.~66]{muirhead}, its
Jacobian is
\[
 dY_1 = 2^{-m} |S|^{-1/2} dS \, (dH), \quad (dH)=\bigwedge_{1\le i<j\le m} h_j^T dh_i,
\]
where $h_i$ is the $i$th column of $H$.  By Theorems 2.1.12 and 2.1.15
in \cite{muirhead},
\[
 \int_{O(m)} (dH) = \frac{2^m\pi^{m^2/2}}{\Gamma_m(\frac{m}{2})},
 \qquad
 \Gamma_m(a) = \pi^{m(m-1)/4}\prod_{i=1}^m \Gamma\left(\frac{2a-i+1}{2}\right),
\]
and we have the joint density of $(Z,S)$ as
\[
 \frac{\pi^{m^2/2}}{(2\pi)^{m^2}\Gamma_m(\frac{m}{2})}
 \exp\left\{-\tfrac{1}{2}\tr S (I + Z Z^T)\right\} |S|^{(m-1)/2} dZ dS.
\]
Moreover, by letting $n=2m$ in the Wishart integral
\[
 \int \exp\left\{-\tfrac{1}{2}\tr S \Sigma^{-1}\right\} |S|^{(n-m-1)/2} dS
= 2^{mn/2} \Gamma_m\left(\frac{n}{2}\right)|\Sigma|^{n/2},
\]
we have
\[
 \int \exp\left\{-\tfrac{1}{2}\tr S \Sigma^{-1}\right\} |S|^{(m-1)/2} dS
= 2^{m^2} \Gamma_m(m)|\Sigma|^{m}.
\]
Hence, the marginal of $Z$ is
\[
 \frac{\pi^{m^2/2}2^{m^2} \Gamma_m(m)}{(2\pi)^{m^2}\Gamma_m(\frac{m}{2})} |I+Z Z^T|^{-m} dZ
= \frac{\prod_{i=1}^m \Gamma(\frac{2m-i+1}{2})}{\pi^{m^2/2} \prod_{i=1}^m \Gamma(\frac{m-i+1}{2})} |I+Z Z^T|^{-m} dZ.
\]

We now restrict our attentions to the case $m=2$.
The density of $Z$ is then
\[
 \frac{1}{\pi^2} \frac{\Gamma(\frac{4}{2})\Gamma(\frac{3}{2})}{\Gamma(\frac{2}{2})\Gamma(\frac{1}{2})} |I+Z Z^T|^{-2} dZ
= \frac{1}{2\pi^2} |I+Z Z^T|^{-2} dZ.
\]
Suppose that $Z$ has real eigenvalues.  For such $Z$, we have the
decomposition $Z=P L P^{-1}$, where $L=\diag(l_1,l_2)$, $l_1>l_2$, and
$P=(p_{ij})_{2\times 2}$ is a non-singular matrix.  Without loss of
generality, we may assume that the eigenvectors
$\begin{pmatrix} p_{1i} \\ p_{2i} \end{pmatrix}$ are unit vectors with
$p_{1i}>0$.  Then the map $Z\mapsto (L,P)$ is one to one.  However, it
will be convenient to also allow $p_{i1}< 0$ in the below
calculations.  In this parameterization, $Z\mapsto (L,P)$ is 1
to $2^2$.

Write
\[
 \begin{pmatrix} p_{1i} \\ p_{2i} \end{pmatrix} = \begin{pmatrix} \cos\theta_i \\ \sin\theta_i \end{pmatrix}, \ \ i=1,2.
\]
The Jacobian of $Z\mapsto (L,P)$ is
\[
 dZ = \frac{(l_1-l_2)^2}{\sin(\theta_1-\theta_2)^2} dL d\theta_1 d\theta_2.
\]
Integrating we find that
\begin{multline*}
 \frac{1}{2\pi^2} |I+Z Z^T|^{-2} dZ = \\
 \frac{1}{2\pi^2}
 \frac{\sin(\theta_1-\theta_2)^2}{\{(1+l_1^2)(1+l_2^2)-(1+l_1 l_2)^2\cos(\theta_1-\theta_2)^2\}^2}\frac{(l_1-l_2)^2}{\sin(\theta_1-\theta_2)^2} dL d\theta_1 d\theta_2
\end{multline*}
over $\theta_1,\theta_2\in[0,2\pi)$.  Dividing by $2^2$, we obtain the
density (not probability density) of $(l_1,l_2)$ as
\[
 \frac{1}{4}\frac{l_1-l_2}{(1+l_1^2)^{3/2}(1+l_2^2)^{3/2}} dL.
\]
Taking the integral over $-\infty<l_2<l_1<\infty$,
\[
 \int_{-\infty<l_2<l_1<\infty}\frac{1}{4}\frac{l_1-l_2}{(1+l_1^2)^{3/2}(1+l_2^2)^{3/2}} dl_1 dl_2 = \frac{\pi}{4}.
\]
This is the integral over the space where $Z$ has real eigenvalues,
and our proof is complete.
\end{proof}

\end{appendix}

\bibliographystyle{amsalpha}
\bibliography{mathias+}

\end{document}